\documentclass[12pt,a4paper]{amsart}
\usepackage[top=35mm, bottom=35mm, left=30mm, right=30mm]{geometry}
\usepackage{xcolor}
\usepackage[colorlinks=true, citecolor=blue]{hyperref}
\usepackage{verbatim}
\usepackage[nobysame]{amsrefs}

\usepackage{mathptmx}

\usepackage{amsthm,amsfonts} 
\newtheorem{thm}{Theorem}[section]

\newtheorem{lem}[thm]{Lemma}
\newtheorem{prop}[thm]{Proposition}
\newtheorem{cor}[thm]{Corollary}

\theoremstyle{definition}
\newtheorem{remark}[thm]{Remark}

\numberwithin{equation}{section}

\DeclareMathOperator{\dimh}{dim_H}
\DeclareMathOperator{\diam}{diam}
\allowdisplaybreaks[4]

\begin{document}

\title{The Hausdorff dimension of multiply Xiong 
chaotic sets}

\author[J. Li]{Jian Li}
\address[J. Li]{Department of Mathematics,
 Shantou University, Shantou, 515063, Guangdong, China}
\email{lijian09@mail.ustc.edu.cn}

\author[J. L\"u]{Jie L\"u}
 \address[J. L\"u]{School of Mathematics, South China Normal University,
 Guangzhou 510631, China}
 \email{ljie@scnu.edu.cn}
 
\author[Y. Xiao]{Yuanfen Xiao}
\address[Y. Xiao]{Department of Mathematics, University of Science and Technology of China, Hefei, 
 230026, Anhui, China}
 \email{xyuanfen@mail.ustc.edu.cn}



 \maketitle
 \begin{abstract}
We construct a multiply Xiong chaotic set 
with full Hausdorff dimension everywhere 
that is contained in some multiply proximal cell   
for the full shift over finite symbols and the Gauss system 
respectively. 
 
 \end{abstract}

\section{Introduction}
By a topological dynamical system, we mean a pair 
$(X,T)$, where $X$ is a metric space and 
$T$ is a continuous self-map of $X$.
A compatible metric on $X$ is denoted by $\rho$.
A subset $S$ of $X$ is said to be scrambled 
if any two distinct points $x,y\in S$ satisfy 
\[
\liminf_{n\to\infty}\rho(T^nx,T^ny)=0  \text{ and }\limsup_{n\to\infty}\rho(T^nx,T^ny)>0.
\]
We say that a topological dynamical system $(X,T)$ is Li-Yorke chaotic if there exists an uncountable scrambled subset of $X$.
As Li-Yorke chaos only requires an uncountable scrambled set, it is natural to ask  how ``large'' a scrambled set can be.
There are two approaches about the size of a set: topological and measure-theoretic.
In \cite{BHS08}, Blanchard, Huang and Snoha studied the topological size 
of scrambled sets extensively, see also the references therein.
In \cite{BL99}, Balibrea and  Jim\'enez L\'opez surveyed the Lebesgue measure of scrambled sets for continuous maps on the interval till then.
In \cite{B09}, Bruin and Jim\'enez L\'opez studied the Lebesgue measure of scrambled sets for $C^2$ and $C^3$  multimodal interval maps $f$ with non-flat critical points.
Recently, in \cite{Li2017} Liu and Li constructed a scrambled set with full Hausdorff dimension  for the Gauss system.

In the study of the complexity of weakly mixing system,
Xiong introduced a kind of chaos in \cite{XX91} and \cite{X91}, 
which is called Xiong chaos nowadays.
We say that a subset $E$ of $X$ containing at least two points is Xiong chaotic 
if  for any subset $A$ of $E$ and any continuous function
$g\colon A\to X$, 
there exists an increasing sequence $\{q_k\}$ of positive integers such that for every $x\in A$,
\[\lim_{k\to\infty}T^{q_k}(x)=g(x).\]
It is clear that every Xiong chaotic set is scrambled.
It is shown in \cite{X91}
that a non-trivial topological dynamical system 
is weakly mixing if and only if it has a $F_\sigma$, $c$-dense, Xiong chaotic set.
In \cite{X95}, Xiong constructed a Xiong chaotic set 
 with full Hausdorff
dimension everywhere in the full shift over finite symbols.
Wu and Tan generalized this result to the full shift over 
countable symbols in \cite{WT07}.

In \cite{BH08}, Blanchard and Huang introduced a local version of weak mixing and showed that if a topological dynamical system has positive topological entropy, then it has ``many'' weakly mixing sets and then ``local'' Xiong chaotic sets.
Recently, Huang et. al generalized this result to $\Delta$-weakly mixing sets \cite{HLYZ17}.
In fact, inspired by the proof of Theorem A in \cite{HLYZ17}, we can proposal the following 
concept of multiply Xiong chaotic set.
We say that a subset $E$ of $X$ containing at least two points is multiply Xiong 
chaotic if 
for any $d\in\mathbb{N}$, any subset $A$ of $E$,  
and continuous functions
$g_j\colon B\to X$ for $j=1,2,\dotsc,d$, there exists an increasing sequence $\{q_k\}$ of positive integers
such that for every $x\in A$,
\[\lim_{k\to\infty}T^{j\cdot q_k}(x)=g_j(x), \quad j=1,2,\dotsc,d.\]
With this terminology, combining with Theorem A and Proposition 3.2 in \cite{HLYZ17}, we have that a non-trivial 
topological dynamical system is $\Delta$-transitive 
if and only if it has a dense, $\sigma$-Cantor, multiply Xiong chaotic set.  In addition, the author in \cite{Liu2019} showed that 
positive topological entropy implies the existence of $\Delta $-weakly mixing subsets for   
finitely generated torsion-free discrete nilpotent group actions.  

The main aim of this paper is to study the 
Hausdorff dimension of multiply Xiong chaotic sets 
in the full shift and the Gauss system. 
In fact, we will study chaotic sets in the 
multiply proximal cell of a point.

Recall that a pair $(x,y)\in X\times X$ is proximal if
$\liminf_{n\to\infty}\rho(T^nx,T^ny)=0$ and 
the proximal cell of a point $x\in X$ is defined 
by
$Prox(x,T)=\{y\in X\colon (x,y)\textrm{ is proximal}\} $.
The structure of the proximal cell plays an 
important role in a topological dynamical system. 
Many results concerning the structure of proximal cell have been studied. For example, Auslander and Ellis proved that every proximal cell contains a minimal point for a compact system \cite{Au60}. 
Moreover, in a weakly mixing system  every proximal cell  is residual  \cite{AK03}. 
In \cite{HW04},  Huang, Shao and Ye proved an 
equivalent statement of an
$\mathcal{F}$-mixing system by the dense $G_\delta $ 
structure of $ Prox(x,T)$ for any $x\in X $ and 
they gave a detailed
description of the proximal cells for $\mathcal{F}$-mixing systems where $ \mathcal{F} $ is a Furstenberg family satisfying some special properties. 

We generalize the proximal to multiply proximal as follows.  
A pair $(x,y)\in X\times X$  is called  multiply proximal if 
for any $ d\in\mathbb{N}$, the pair $ (x,y) $ satisfies 
$\liminf_{n\to\infty}\max_{1\leq j\leq 
d}$  $\rho(T^{j\cdot n}(x), T^{j\cdot n}(y))=0$
and the multiply proximal cell of a point $ x\in X$ 
 is denoted by $ MProx(x,T)=\{y\in X\colon (x,y)\textrm{ 
 is multiply proximal} \} $. It is clear that $MProx(x,T) $ is a subset of $Prox(x,T)$. 
We will study the multiply Xiong chaos in multiply proximal cells for some special systems. 
Now we are ready to state the main results  of 
this paper as follows:

\begin{thm}\label{main1}
Let $ N\geq 2 $ be a positive integer. 
In the full shift $(\Sigma_N,\sigma)$ over $N$-symbols,
for every $z\in \Sigma_N$, the multiply proximal 
cell of $z$ contains
a  multiply Xiong chaotic sets 
with full Hausdorff dimension everywhere.
\end{thm}

The Gauss map $T\colon [0,1)\to [0,1)$ is defined by
 $T(x)=\frac{1}{x}-\bigl\lfloor \frac{1}{x}\bigr\rfloor$ for $x\in (0,1)$ and $T(0)=0$, where $ \lfloor z\rfloor$ denote  the greatest integer less than or equal to  $z$. The restriction of $ T $ on $ [0,1)\setminus\mathbb{Q} $  is called the Gauss system. 
\begin{thm}\label{main2}
In the Gauss system $([0,1)\setminus\mathbb{Q},T)$,
for every   $z$ in $ [0,1)\setminus\mathbb{Q} $, the multiply proximal 
cell of $z$ contains a multiply Xiong chaotic sets with full Hausdorff dimension everywhere.
\end{thm}

The paper is organized as follows. In Section 2, we  introduce some preliminaries.   
Section 3  is devoted to proving Theorem 
~\ref{main1}.
In Section 4, Theorem~\ref{main2}  is proved 
and  we discuss the properties of the scrambled set for 
the Gauss system.

\section{Preliminaries}
In this section, we present some basic notations,  definitions
and results that will be used later.

\subsection{Upper density}
Denote $\mathbb{N}$ the set of positive integers.
For a finite subset $F$ of $\mathbb{N}$, 
denote the cardinality of $F$ by $\#(F)$.
For a subset $A$ of $\mathbb{N}$,
the upper density of $A$ is defined as 
\[\overline{D}(A)=\limsup\limits_{m\to\infty}
\tfrac{1}{m}\#(A\cap[1,m]).\]
For a strictly  increasing sequence in $\mathbb{N}$,
we can view it as an infinite subset of $\mathbb{N}$
and define the upper density of this sequence as the one of the infinite subset.

\subsection{Hausdorff dimension}
In a metric space $ (X,\rho) $, for a subset $ A $ of $ X $,  a real number  
$ \delta >0 $,  and $ s\geq 0 $, 
define   
\[
\mathcal{H}^s_{\delta}(A)=\inf\biggl\{\sum_{i\geq 1}
 \diam(U_i)  ^s: A\subset \bigcup_{i\geq 1}U_i  \text{ and }
 \diam(U_i) <\delta \text{ for any }i\geq 1  \biggr\}   
\]  
where $ \diam(\cdot)  $ denotes the diameter 
of a set.   
The $ s $-dimension Hausdorff measure of $ A $ is given by
\[
\mathcal{H}^s(A)=\lim\limits_{\delta\to 
0} \mathcal{H}^s_{\delta}(A) 
\]
and the Hausdorff dimension of  $ A $  is 
\[
\dimh(A)=\biggl\{\begin{array}{ll}  
\inf\{s> 0:\mathcal{H}^s(A)=0 \}, & \text{ if } 
\{s> 0:\mathcal{H}^s(A)=0 \}\neq \emptyset;\\
+\infty, & \text{ otherwise. }
 \end{array}\biggr.
 \]
The basic knowledge about 
Hausdorff dimension can be found in \cite{Fa2003}, which we 
refer the reader to.
We say that a subset $ C $ of $ X $ has full 
Hausdorff dimension 
everywhere if $ \dimh(C\cap U)=\dimh(X) $ for any 
non-empty open subset $ U $ of $ X $ .

Let  $ \alpha $ 
be a positive real number. We say that  a map $ 
f:X\to \mathbb{R} $ satisfies the locally $ 
\alpha $-H\"older condition if there exist a 
real number $ r>0 $ and a constant $ c>0 $ such 
that $ |f(x)-f(y)|\leq c \bigl(\rho(x,y)\bigr)^{\alpha} $ holds for  any $ x,y\in X $ with $ \rho(x,y)<r $. 

The following well known lemma 
 can be easily deduced from the definitions 
of Hausdorff dimension and the locally $ \alpha $-H\"older condition. 
 \begin{lem}\label{holder}
 Let $(X,\rho)$ be a metric space and 
 $s, \alpha >0 $ be real numbers.
 If a map $f:X\to \mathbb{R} $ satisfies the locally 
 $ \alpha $-H\"older condition, then  
 $\mathcal{H}^s(f(X))\leq c^s \mathcal{H}^{s\cdot\alpha }(X) $, where $ c $ is the constant in the definition of the locally 
 $ \alpha $-H\"older condition. 
 Moreover, $\alpha\dimh(f(X))\leq \dimh(X) $. 
\end{lem} 

\subsection{Ergodic theory}  
For a probability space $ (X,\mathcal{B},\mu) $, 
a transformation $ T\colon X\to X $ is 
called measure-preserving if $T^{-1}(\mathcal{B})\subset \mathcal{B}$ and $ \mu(A)=\mu(T^{-1}(A))$ for any $A $ in the
 $ \sigma $-algebra $ \mathcal{B} $. 
In this case, the quadruple $(X,\mathcal{B},\mu,T)$ is called a measure-preserving system.
A measure-preserving transformation $(X,\mathcal{B},\mu,T)$
is called ergodic if the only members $ A $ of $ 
\mathcal{B} $  with $ T^{-1}(A)=A $ satisfies $ 
\mu(A)=0 $ or $\mu(A)=1 $, weakly mixing if $ (X\times X, \mathcal{B}\times \mathcal{B}, 
\mu\times \mu ,T\times T) $ is ergodic, and strongly 
mixing if $ \lim\limits_{n\to\infty} 
m(T^{-n}(A)\cap B)=m(A)m(B) $ for any $ A,B\in\mathcal{B} $.  
It is obviously that every 
strongly mixing transformation is weakly mixing 
and every weakly mixing transformation is ergodic.

We say that  a measure-preserving system  
$(X,\mathcal{B},\mu,T)$ is exact
if for each $B$ in the tail $\sigma$-algebra  
$\bigcap_{n\in\mathbb{N}}
T^{-n}(\mathcal{B})$, either $\mu(B)$ or $\mu(X\setminus B)$  is zero.
The exactness was introduced by Rokhlin in 
\cite{R61} and he obtained  the following result.
\begin{prop}[\cite{R61}] \label{prop:exact}
\begin{enumerate}
 \item If a measure-preserving system  $(X,\mathcal{B},m,T)$ is exact, then it is strongly mixing.
 \item A  measure-preserving system $(X,\mathcal{B},m,T)$ is exact  if and only if
 $\lim_{n\to\infty} m(T^nA)=1$ for every $A\in\mathcal{B}$
 with $m(A)>0$ and $T^nA\in\mathcal{B}$ for every 
 $n\geq 1$.
\end{enumerate}
\end{prop}

\section{multiply Xiong chaotic sets in full 
shift over finite symbols}

In this section, we construct a multiply
Xiong chaotic set with full Hausdorff dimension 
everywhere for the full shift over finite symbols, 
which proves Theorem~\ref{main1}.

\subsection{The full shift over finite symbols}
Let $N$ be an integer with $N\geq 2$. We endow the finite 
set $\{1,2,\dotsc,N\}$ with the discrete topology  and  
the product space
$ \Sigma_N=\prod_{n=1}^\infty\{1,2,\dotsc,N\} $
is compact and metrizable. 
A compatible metric $\rho$ on $\Sigma_N$ can be 
defined as follow.
For any $x=x_1x_2\dotsb$ and
$y=y_1y_2\dotsb\in\Sigma_N$,
\[
\rho(x,y)=\begin{cases}
0,& \text{ if }x=y,\\
\frac{1}{N^k}, &\text{ if }x\neq y \text{ and }
k=\min\{n\in\mathbb{N}\colon x_n\neq y_n\}-1. 
\end{cases}
\]
The shift map $\sigma\colon\Sigma_N\to\Sigma_N$ 
is defined as 
$\sigma(x)=x_2x_3\dotsb$ for any 
$x=x_1x_2\dotsb\in\Sigma_N$.
It is clear that $\sigma$ is continuous.
The dynamical system $(\Sigma_N,\sigma)$ is 
called the full shift over $ N $ symbols.

Let $ n\geq 1 $.
Each element in  $\{1,2,\dotsc,N\}^n$ is called a word  
with length $n$ and denote the set consisting of all the words by $ \mathcal{N}^*$.  
For a point $x=x_1x_2\dotsb  \in\Sigma_N$,
denote its prefix with length $n$ by 
$x[1,n]=x_1x_2\dotsb x_n\in \{1,2,\dotsc,N\}^n$.
We also use $ x[n] $ to represent the number at 
the $ n $-position of $ x $ for any $ n\geq 1 $. 
 We  say that the number $ n $ is the position of the word $ x[n,m] $ appeared in the point $ x $ for any $ m\geq n $.
The cylinder generated by the word $x_1x_2\dotsb 
x_n$ is the set consisting of the points with the same word $x_1x_2\dotsb  x_n$ as its prefix and denote it by  $ [x_1x_2\dotsb  x_n] = \{y\in\Sigma_{N}: 
y_i=x_i \text{ for } i=1,2,\dotsc , n \} $. 
Clearly, a cylinder is both open and closed. 
For $u=u_1u_2\dotsb u_n$ and $v=v_1v_2\dotsb v_m$,
denote by $ uv  $ or  $ u\sqcup v $
the concatenation of $u$ and $v$, that is 
$ uv=u\sqcup v =u_1u_2\dotsb u_nv_1v_2\dotsb v_m $. Since $ u\sqcup v $ may be different from $ v\sqcup u $, we require that the  symbol ``$ \bigsqcup_{1\leq i\leq n} u_i $ '' means $ u_1\sqcup u_2\sqcup\dotsb \sqcup u_n $. 
It is not hard to show that with the metric $ \rho $, any cylinder has full Hausdorff dimension, which is equal to $ 1 $.

The following result was proved in \cite[Lemma 3]{X95},
which is a key estimation of  the Hausdorff dimension of a subset of $\Sigma_N$.
\begin{lem}\label{weishu}
Let $A=\{a_n\}_{n=1}^\infty$ be a sequence of strictly increasing positive numbers.
Define a map $ \Gamma_{A}:\Sigma_{N}\to\Sigma_{N} $, 
$ \Gamma_{A}(x)=x_1x_2\dotsc  x_{a_1-1}x_{a_1+1}\dotsc  
x_{a_n-1}x_{a_n+1}\dotsc $.
Let $Y$ be a subset of $\Sigma_N$.
If the upper density of $A$ is  $\lambda$
and $\mathcal{H}^1(\Gamma_{A}(Y))>0$, 
then $ \dimh(Y)\geq 1-\lambda$. 
\end{lem}

\subsection{Proof of Theorem~\ref{main1}}
Now we construct a multiply Xiong chaos set step by step.
Note that the main idea comes from \cite{X95}, but we should amend the method to multiply Xiong chaotic case. 
We first define a map $\Delta_N\colon\Sigma_N\to\Sigma_N$ to construct a subset, then we show that this subset meets the requirement. 

Fix a point $z=z_1z_2\dotsc\in\Sigma_N$. 
For each $ n\in\mathbb{N} $, let 
 $l_n=(N^n)^{N^n} $. We 
list all the self-maps  on $\{1,2,\dotsc ,N\}^n $  as 
$ \varphi_{1}^{(n)},\varphi_2^{(n)},\dotsc, \varphi_{l_n}^{(n)} $.  
Fix a point  $ x=x_1x_2\dotsb $ in $ \Sigma_{N}$.
To define $\Delta_N(x)$, we first use these maps $\varphi_i^{(n)}$ 
to construct a series of words functioning as the chaotic part. 
For any integer  $1\leq i\leq {l_n}^{l_n}  $ and 
$ 1\leq j\leq l_n-1 $, we  choose a word $ U_{i,j}^{(n)} $ in $ \mathcal{N}^* $ that  will be specialized later.
We list the elements in $\{1,2,\dotsc,l_n\}^{l_n}$
as $\mathbf{p}_i^{(n)}=\bigl(p_{i,1}^{(n)},p_{i,2}^{(n)},\dotsb,p_{i,l_n}^{(n)}\bigr)$, $1\leq i\leq {l_n}^{l_n}  $. 

First, we deal with the case when $ k=1 $.  
Define a word as follow, 
\[
\Phi^{(1)}(x_1)= \bigsqcup_{1\leq i\leq {l_1}^{l_1}  } 
\varphi_{p_{i,1}^{(1)}}^{(1)}(x_1)U_{i,1}^{(1)}
\varphi_{p_{i,2}^{(1)}}^{(1)}(x_1)U_{i,2}^{(1)}\dotsc U_{i,l_1-1}^{(1)}
 \varphi_{p_{i,l_1}^{(1)}}^{(1)}(x_1). 
 \]
Second, when $ k=2 $, define a word  
\[
\Phi^{(2)}(x_1x_2)= \Phi^{(1)}(x_1) \sqcup  
\bigsqcup_{1\leq i\leq {l_2}^{l_2}  } 
\varphi_{p_{i,1}^{(2)}}^{(2)}(x_1x_2)U_{i,1}^{(2)}
\varphi_{p_{i,2}^{(2)}}^{(2)}(x_1x_2)U_{i,2}^{(2)}\dotsc
U_{i,l_2-1}^{(2)}
\varphi_{p_{i,l_2}^{(2)}}^{(2)}(x_1x_2) .
\]
Continuing this process, for each $n\in\mathbb{N}$, we can define  a word   
\begin{align*}
\Phi^{(n)}(x_1x_2\dotsb x_n)= &
\Phi^{(n-1)}(x_1x_2\dotsb x_{n-1})   \sqcup 
\bigsqcup_{1\leq  i\leq  {l_n}^{l_n}  } 
\Bigl(\varphi_{p_{i,1}^{(n)}}^{(n)}(x_1x_2\dotsb 
x_n)U_{i,1}^{(n)}  \\
& \sqcup \varphi_{p_{i,2}^{(n)}}^{(n)}(x_1x_2\dotsb 
x_n)U_{i,2}^{(n)}\dotsc U_{i,l_{n}-1}^{(n)}
\varphi_{p_{i,l_n}^{(n)}}^{(n)}(x_1x_2\dotsb x_n)\Bigr).
\end{align*}  
Denote the length of $ \Phi^{(n)}(x_1x_2\dotsb x_n)  $ by  $ c_n$.  
  
Next, we start to define another series of 
words in order to meet the requirement of the  proximal part. Set $ s_n=(l_n)^{l_n+1} $  for 
any $ n\geq 1 $.  
When $ k=1 $,   define a word 
\[
\Delta_{N}^{(1)}(x)=V_{1,1}z[s_1+1]V_{1,2}z[1]
\Phi^{(1)}(x_1).
\]
where $ V_{1,1} $ and $ 
V_{1,2} $ are two words that will be defined later. 
We determine the word $\Delta_{N}^{(1)}(x)$ by the following rules:
\begin{enumerate}
 \item the length  of  $ V_{1,1}  $ is equal to $ s_1 $;
 \item the length of $  V_{1,2} $ is equal to $ s_1-1 $;
 \item we choose 
 a proper length of each $ U_{i,j}^{(1)} $ such that  $ 
 r_{i,j}^{(1)}-1=j(r_{i,1}^{(1)} -1) $  
 for $  1\leq  i\leq {l_1}^{l_1} $ 
 and  $ 1\leq j\leq l_1-1 $,
 where  $ 
 r_{i,j}^{(1)}=r_{i,j}^{(1)}(s_1) $ is the 
 position of the word 
 $ \varphi_{p_{i,j}^{(1)}}^{(1)}(x_1) $ appeared in 
 $ \Delta_N^{(1) }(x) $.
 \item we specialize the words $V_{1,j}$ and $U_{i,j}^{(1)}$ such that
 the word 
 \[Q_1=  V_{1,1}V_{1,2}\sqcup \Biggl(\bigsqcup_{1\leq i\leq l_1^{l_1}}\bigsqcup_{1\leq j\leq l_1-1} U_{i,j}^{(1)} \Biggr) \]
  is an initial segment of $x$, that is $Q_1=x[1,q_1]$, where $q_1$ is the length of $Q_1$. 
\end{enumerate}

Let $ t_1=s_1+2 $.  Denote by  
$ D_1=\{a_1^{(1)}<a_2^{(1)}<\dotsb< a_{t_1}^{(1)}\} $ 
the set of the positions of $ z[s_1+1] $,  $ z[1] $, and 
 $ \varphi_{p_{i,1}^{(1)}}^{(1)}(x_1) $,
$  \varphi_{p_{i,2}^{(1)}}^{(1)}(x_1)  $,  $ \dotsc  $, 
$ \varphi_{p_{i,l_1}^{(1)}}^{(1)}(x_1) $ for $ 1\leq i\leq {l_1}^{l_1} $ appeared in $ \Delta_N^{(1)}(x) $. 
Note that  the position of $ z[s_1+1] $ appeared 
in $ \Delta_{N}^{(1)}(x_1) $ is $ s_1+1 $ and the 
position of $ z[1] $ appeared in $ 
\Delta_{N}^{(1)}(x_1) $ is $ 2s_1+1 $.  We estimate the density of the set $ D_1 $ as follow. 
\begin{enumerate}
\item[(i)] If an integer $ m $ satisfies $ s_1+1\leq m\leq 
2s_1+1 $, 
then 
\[ \frac{ \#\{j\leq m: j\in D_1 \}  }{m }<\frac{2}{m 
}<\frac{2}{s_1+1}\leq 1 . \]

\item[(ii)] If $ 2s_1+1<m\leq c_1+2s_1+1  $, then
\[ \frac{ \#\{j\leq m: j\in D_1\}  }{m}<\frac{t_1}{2s_1+1}<1 .\]
\end{enumerate}

When $ k=2 $,   
define a word 
\[
\Delta_{N}^{(2)}(x)=\Delta_N^{(1)}(x) 
V_{2,1}z[A_1^{(2)}+1,A_1^{(2)}+2]V_{2,2} 
z[A_1^{(2)}+1,A_1^{(2)}+2] 
V_{2,3} z[1,2]
V_{2,4}z[1,2]\Phi^{(2)}(x_1x_2) 
\]
where $A_i^{(2)}$ is a positive integer and $ V_{2,i} $ is a word for $ 1\leq i\leq 4 $, which will be defined later.
We determine the word $\Delta_{N}^{(2)}(x)$ by the following rules:
\begin{enumerate}
\item $ A_1^{(2)}>2(t_1+8) $ and $ A_4^{(2)}>t_1+t_2-3 $, where $ t_2=2s_2+8 $;
 \item the length  of $ \Delta_{N}^{(1)}(x)V_{2,1} $ is equal to  $ A_1^{(2)} $ and the length of 
 $ V_{2,i}  $ is  $ A_i^{(2)} $ for $  2\leq i\leq 4 $; 
 \item  $2+A_2^{(2)}=A_1^{(2)}$ and $A_1^{(2)}+2+A_2^{(2)}+2+A_3^{(2)}=2+A_4^{(2)}$;

 \item  we 
 choose  a proper length of each $U_{i,j}^{(2)} $ such that  
 $r_{i,j}^{(2)}-1=j(r_{i,1}^{(2)}-1) $  
 for $  1\leq  i\leq {l_2}^{l_2} $ 
 and  $ 1\leq j\leq l_2-1 $,
where $r_{i,j}^{(2)}=r_{i,j}^{(2)}(A_1^{(2)},A_4^{(2)})$ 
 is the position of the word 
 $ \varphi_{p_{i,j}^{(2)}}^{(2)}(x_1x_2) $ appeared 
 in  $ \Delta_N^{(2) }(x) $;
 \item we specialize the words $V_{2,j}$ and $U_{i,j}^{(2)}$ such that 
 the word  
 \[Q_2=Q_1\sqcup V_{2,1}V_{2,2}V_{2,3}V_{2,4}\sqcup \Biggl(\bigsqcup_{1\leq i\leq l_2^{l_2}}\bigsqcup_{1\leq j\leq l_2-1} U_{i,j}^{(2)} \Biggr) \] 
 is an initial segment of $x$.
 
\end{enumerate}
We denote by  
$ D_2=\{a_1^{(2)}<a_2^{(2)}<\dotsb< a_{t_2}^{(2)}\} $ 
the set of the positions of 
$ z[A_1^{(2)}+1] $, $ z[A_1^{(2)}+2] $,  
$ z[1] $, $ z[2] $    
$ \varphi_{p_{i,1}^{(2)}}^{(2)}(x_1x_2)[1]$, $ \varphi_{p_{i,1}^{(2)}}^{(2)}(x_1x_2)[2]$
$ \varphi_{p_{i,2}^{(2)}}^{(2)}(x_1x_2)[1] $, $ \varphi_{p_{i,2}^{(2)}}^{(2)}(x_1x_2)[2] $,  $ \dotsc $ , 
$\varphi_{p_{i,l_2}^{(2)}}^{(2)}(x_1x_2)[1] $, 	$\varphi_{p_{i,l_2}^{(2)}}^{(2)}(x_1x_2)[2] $ for $ 1\leq i\leq l_2^{l_2} $ appeared in 
$ \Delta_N^{(2)}(x) $.  
We estimate the density of set $ D_1\cup D_2 $ as follow.
\begin{enumerate}
\item[(i)] If $ c_1+2s_1+1<m\leq A_1^{(2)}  $, then 
\[ \frac{ \#\{j\leq m: j\in D_1\cup D_2\}  }{m 
}<\frac{t_1}{2s_1+1+c_1}<1. \]

\item[(ii)] If $ A_1^{(2)}<m\leq 
A_1^{(2)}+2+A_2^{(2)}+2+A_3^{(2)}+2+A_4^{(2)}+2 
$, then \[ \frac{ \#\{j\leq m: j\in D_1\cup D_2\}  }{m}<\frac{t_1+8}{A_1^{(2)} }<\frac{1}{2} . \]
          
\item[(iii)] If   $ A_1^{(2)}+2+A_2^{(2)}+2+A_3^{(2)}+2+A_4^{(2)}+2< 
m\leq A_1^{(2)}+2+A_2^{(2)}+2+A_3^{(2)}+2+A_4^{(2)}+c_2 $, then 
\[ \frac{ \#\{j\leq m: j\in D_1\cup  D_2 \}  }{m 
}<\frac{t_1+t_2}{A_1^{(2)}+2+A_2^{(2)}+2+A_3^{(2)}
+2+A_4^{(2)}+2 }<\frac{1}{2}. \]
\end{enumerate}

Assume that we have completed the construction in the case $ k=n-1 $. We carry on  the case  $ k=n $ as follows.
Define a word 
\begin{align*}
\Delta_{N}^{(n)}(x)=&\Delta_N^{(n-1)}(x) 
V_{n,1}z[A_1^{(n)}+1,A_1^{(n)}+n]V_{n,2} 
z[A_1^{(n)}+1,A_1^{(n)}+n] \sqcup\dotsb \\
&\sqcup 
V_{n,n} z[A_1^{(n)}+1,A_1^{(n)}+n] \sqcup 
V_{n,n+1} z[1,n]
V_{n,n+2}z[1,n]\sqcup\dotsb\sqcup V_{n,2n}z[1,n]\\
&\sqcup \Phi^{(n)}(x_1x_2\dotsb x_n) 
\end{align*} 
where  $A_i^{(n)} $ is a positive integer and  $ V_{n,i} $ is a word for $ 1\leq i\leq 2n $, which will be defined later. 
We determine the word $\Delta_{N}^{(n)}(x)$ by the following rules:
\begin{enumerate}
\item[(R1)]  $ 
A_1^{(n)}>n(t_{n-1}+2n^2) $ 
and $ A_{2n}^{(n)}>t_{n-1}+t_n-(n+1) $, where $ t_n=n\cdot s_n+2n^2 $;
\item [(R2)] the length of $ \Delta_{N}^{(1)}(x)V_{n,1} $ is equal to  $ A_1^{(n)} $ and the length of  
 $ V_{n,i}  $ is  $ A_i^{(n)} $ for $2\leq i\leq 2n$; 
 
\item[(R3)] $A_1^{(n)}=n+A_2^{(n)}=\dotsb=n+A_n^{(n)}$ and $
 A_1^{(n)}+A_2^{(n)}+\dotsb+A_n^{(n)}+n^2+A_{n+1}^{(n)}
 =n+A_{n+2}^{(n)}=n+A_{n+3}^{(n)}
 =\dotsb=n+A_{2n}^{(n)}$;
 
\item[(R4)] we 
 choose  a proper length of each $ 
 U_{i,j}^{(n)} $such that  $ 
 r_{i,j}^{(n)}-1=j(r_{i,1}^{(n)} -1) $  
 for $  1\leq  i\leq {l_n}^{l_n} $ 
 and  $ 1\leq j\leq l_n-1 $,
where $
 r_{i,j}^{(n)}=r_{i,j}^{(n)}(A_1^{(n)},A_{2n}^{(n)})
 $ is the position of the word 
 $ \varphi_{p_{i,j}^{(n)}}^{(n)}(x_1x_2\dotsb x_n) $ 
 appeared  in    $ \Delta_N^{(n) }(x) $;
 
 \item [(R5)] we specialize the words $V_{n,j}$ and $U_{i,j}^{(n)}$ such that 
 the word 
 \[ Q_n=Q_{n-1}\sqcup   \bigsqcup_{j=1}^{2n}V_{n,j} \sqcup  \Biggl(\bigsqcup_{1\leq i\leq l_n^{l_n}}\bigsqcup_{1\leq j\leq l_n-1} U_{i,j}^{(n)} \Biggr)  \]
   is an initial segment of $x$.
\end{enumerate}
We denote by  $ D_n=\{a_1^{(n)}<a_2^{(n)}<\dotsb< a_{t_n}^{(n)}\} 
$ the set of the 
positions of $ z[A_1^{(n)}+1] $, $ z[A_1^{(n)}+2] 
$, $ \dotsc  $, $ z[A_1^{(n)}+n] $,   $ 
z[1] $, $ z[2] $,$ \dotsc $, $ z[n] $  and 
$ \varphi_{p_{i,1}^{(n)}}^{(n)}(x_1x_2\dotsb x_n)[1]$, $ \varphi_{p_{i,1}^{(n)}}^{(n)}(x_1x_2\dotsb x_n)[2]$, $ \dotsc $, $ \varphi_{p_{i,1}^{(n)}}^{(n)}(x_1x_2\dotsb x_n)[n]$,
$ \varphi_{p_{i,2}^{(n)}}^{(n)}(x_1x_2\dotsb x_n)[1] $,
$ \varphi_{p_{i,2}^{(n)}}^{(n)}(x_1x_2\dotsb x_n)[2] $,
$ \dotsc $, $ \varphi_{p_{i,2}^{(n)}}^{(n)}(x_1x_2\dotsb x_n)[n] $,	$ 	\dotsc $, $\varphi_{p_{i,l_n}^{(n)}}^{(n)}(x_1x_2\dotsb x_n)[1] $, $\varphi_{p_{i,l_n}^{(n)}}^{(n)}(x_1x_2\dotsb x_n)[2] $, $ \dotsc $, $\varphi_{p_{i,l_n}^{(n)}}^{(n)}(x_1x_2\dotsb x_n)[n] $  for $ 1\leq i\leq {l_n}^{l_n} $
appeared in $ \Delta_N^{(n)}(x) $.    
We estimate the density of set $ \bigcup_{i=1}^{n}D_i  $ as follow.
\begin{enumerate}
\item [(i)] If $1\leq m\leq A_1^{(n)}  $, then 
\[ \frac{ \#\{j\leq m: j\in  \bigcup_{i=1}^n 
D_i\}  }{m }\leq \frac{1}{n-1} . \]

\item [(ii)] If $ A_1^{(n)}<m\leq 
A_1^{(n)}+A_2^{(n)}+\dotsb+A_{2n}^{(n)}+2n^2 $, then \[
\frac{ \#\{j\leq m: j\in \bigcup_{i=1}^n D_i\} \}  }{m }<\frac{t_{n-1}+2n^2 }{A_1^{(n)} }<\frac{1}{n} . \]

\item [(iii)] If   $ 
A_1^{(n)}+A_2^{(n)}+\dotsb+A_{2n}^{(n)}+2n^2< m\leq A_1^{(n)}+A_2^{(n)}+\dotsb+A_{2n}^{(n)}+2n^2+c_n $, then \[
\frac{ \#\{j\leq m: j\in \bigcup_{i=1}^n D_i\}  
\}}{m}<\frac{t_{n-1}+t_{n}}{A_1^{(n)}+A_2^{(n)}+
\dotsb+A_{2n}^{(n)}+2n^2 }<\frac{1}{n}. \]
\end{enumerate}
By induction the sequence $\{\Delta_N^{(n)}(x)\}_{n=1}^{\infty} $ is defined,
now we define 
$ \Delta_{N}(x)=\lim_{n\to\infty}\Delta_{N}^{(n)}(x)0^\infty$.
It is easy to see that $ \Delta_N $ is   
injective and continuous.

For any $ n\geq 1 $ and  $ v=v
_1v_2\dotsb v_n\in\{1,2,\dotsc,N\}^n $, 
define a map $ \theta_v: \Sigma_{N}\to 
\Sigma_{N}$ by $ 
\theta_v(x)=v_1v_2\dotsb v_n 
x_{n+1}\cdots $ for any $ x=x_1x_2\dotsb 
\in\Sigma_{N} $. 
It is clear that $\theta_v $ is a continuous and open map.
Numerate the countable set
$ \bigcup_{k=1}^{\infty}\{1,2,\dotsc,N\}^k $
as $ \{w_n:n\geq 1\}$.
For any $n\geq 1$, 
set  $ B_n=[222\dotsb21]\in\{1,2,\dotsc,N\}^{n+1} $ and 
$   C_{n}=\theta_{w_n}\circ \Delta_N(B_n) $.  Then we make the countable union as   
$    C= \bigcup_{n= 1}^{\infty} C_{n }  $ and    
 $B=\bigcup_{n= 1}^{\infty}\Delta_{N}(B_n)  $.   
Based on the fact that $  \Delta_{N}(B_n)  $ is pairwise disjoint closed for any $ n\geq 1 $ (because both $ \{B_n\} $ and   $ 
\{C_{n } \} $ are two sequences of 
disjoint closed subsets),  we can define a map $ 
\xi :B\to C $ such that $ 
\xi|_{\Delta_{N}(B_n)}  =  
\theta_{w_n}|_{\Delta_{N}(B_n)}$ for any $ n\geq 
1$.   Note that $
\Delta_N(B_n) $  is a  
closed subset of $ B $ for any $ 
n\geq 1 $, since $
\Delta_N(B_n) $  is a countable intersection of closed subsets of $ \Sigma_{N} $. 
And $ \theta_{w_n} $ is 
continuous for any $ n\geq 1 $, the map $ \xi $ is continuous.

We now turn to show that  $C$ is 
contained in the multiply proximal cell of the point $ z $. 
For any $ d\in\mathbb{N} $ and  $ y \in C$, there exists $n\in\mathbb{N}$ such that 
$y\in \theta_{w_n}\circ\Delta_{N}(B_n) $.  
According to (R2)  and (R3),  we know 
that $ \lim_{k\to\infty}\rho(\sigma^{j\cdot A_1^{(k)} }(z), 
\sigma^{j\cdot A_1^{(k)} }(y) )=0 $  for $j=1,2,\dotsc,d$. 
Hence, $ C $ is 
contained in the multiply proximal cell of $z$.  
Furthermore, we can also obtain that 
$\lim_{k\to\infty} \diam(\{z\}\cup 
\sigma^{k+A_{2k}^{(k)}}(C)) =0$.

Next, we show that  $ C $  has full Hausdorff dimension everywhere. 
First, it is not hard to show  that $\dimh(\Sigma_{N})=\dimh(B_n) 
=1 $  for any $ n\geq 1 $. 
Second, let $ D=\bigcup_{i=1}^{\infty}D_i $. 
By (R5) it is easy to see that $ \Gamma_{D}\circ \Delta_{N} $ is identity.
Third, we claim that the density of the set $ \bigcup_{i=1}^{\infty}D_i $ 
is zero. The reason is that for any $ \frac{1}{k} $ with $ k $ being an integer larger than $ 2 $, there exists an integer $  A_1^{(k)} $ such that for any $ m>A_1^{(k)} $ we can find some integer $ n \geq k $ with $ A_1^{(n)}<m\leq A_1^{(n+1)}  $ such that  
\[ 
\frac{\#(\bigcup_{i=1}^{\infty}D_i\cap [1,m] ) }{m}<\frac{ \#(\bigcup_{i=1}^{n+1}D_i\cap [1,m])     }{m}<\frac{1}{n}\leq \frac{1}{k}. 
\]
Fourth, by Lemma~\ref{weishu}, we know that
$ \dimh(\Delta_{N}(B_n))=1 $ because $ \mathcal{H}^{1}(\Gamma_{D}\circ\Delta_{N}(B_n))= \mathcal{H}^{1}(B_n)>0  $ and the density of $ D $ is zero. 
As the map $\theta_{w_n}  $ only changes 
the prefix of the points in $ B_n $ with length $n $, it turns out that  $\dimh(\theta_{w_n}(\Delta_{N}(B_n)))=\dimh(\Delta_{N}(B_n))=1 $ for any $ n\geq 1 $.
At last, for any non-empty open subset $ U $ of $ \Sigma_{N} $, there exists some $ w_n $ such that $ [w_n]\subset U $. 
Since $ C_n=\theta_{w_n}\circ\Delta_{N}(B_n) \subset  U$,  it is clear that $\dimh(C\cap U)\geq \dimh(C_{n}\cap U )=\dimh(C_{n}  )
= \dimh(\theta_{w_n}(\Delta_{N}(B_n)))=1 $.

Finally, it remains to show that $ C $ is  multiply Xiong chaotic. 
As $\xi\colon B\to C$ is continuous, it is sufficient to show that 
$B$ is a multiply Xiong chaotic set. 
Let $ E $ be a subset of $ B $ and 
$g_j:E\to \Sigma_{N}$ be a continuous map for $ j=1,2,\dotsc,d $.
For any $ x\in E $, any $ j=1,2,\dotsc,d $,  and any $ k\geq 1 $,  
define an integer 
\begin{align*}
\psi_{k,j}(x)=\max\{0\leq i\leq k\colon & \text{there exists } 
Y_{i,j}\subset \{1,2,\dotsc,N\}^i \text{ such that }\\ 
&x[1,k]\cap E \subset  g_j^{-1}([Y_{i,j}]) \} .
\end{align*} 
If $ \psi_{k,j}(x)>0 $,  there exists a unique word $ Y_{k,j}(x) $ 
in $\{1,2,\dotsc,N\}^{\psi_{k,j}(x)} $ such that 
$  x[1,k]\cap E\subset  g_j^{-1}([Y_{k,j}(x)]) $,  
which also means  $ g_j(x)\in [Y_{k,j}(x)] $ 
for $ j=1,2,\dotsc,d $. 
We list two useful results as follows,
\begin{enumerate}
\item[(P1)] Fix a positive integer $ k $, if  $ x,y\in E $  with $ x[1,k]=y[1,k]  $, then $ \psi_{k,j}(x)=\psi_{k,j}(y) $. Furthermore, if $ \psi_{k,j}(x)=\psi_{k,j}(y)>0 $, then $ Y_{k,j}(x)=Y_{k,j}(y) $ for $ j=1,2,\dotsc,d $.

\item[(P2)] $\lim_{k\to\infty}\psi_{k,j}(x)=\infty $ for 
$ j=1,2,\dotsc,d $. From the continuity of $ g_j $ for $ j=1,2,\dotsc,d $ and the fact that $  \lim_{k\to\infty}\diam(x[1,k]) =0 $, it is clear  that $  \lim_{k\to\infty} \diam(g_j(x[1,k]\cap E) )  =0$.  
In other words, for any integer  $ M>0 $,  there exists $ M_1>M $ such that  $ \diam( g_j(x[1,k]\cap E))  <\frac{1}{N^M} $ for $j=1,2,\dotsc,d $ and any $ k\geq M_1 $.  
It  implies that there exists  a word $ Y_{k,j} $ with length $ M $ 
such that  $ x[1,k]\cap E \subset  g_j^{-1}([Y_{k,j}]) $ 
for $ j=1,2,\dotsc,d $. 
So,  for this $ M $, $ M_1 $, and any $ k\geq M_1 $, 
it has that $ \psi_{k,j}(x)\geq M $,  which means  $\lim_{k\to\infty}\psi_{k,j}(x)=\infty $ for 
$ j=1,2,\dotsc,d $. 
\end{enumerate}

By (P2), there exists $ m(x)>0 $ such that  $ \psi_{k,j}(x)>0 $ for $j=1,2,\dotsc,d $ and any $ k\geq m(x) $. Therefore, for a sufficient large $ k $ we can  
 find the cylinder $ Y_{k,j}(x) $ and in fact $ 
 \lim_{k\to\infty} \diam(Y_{k,j}(x)) =0 $ for $ j=1,2,\dotsc,d $. 
 
If $ \psi_{k,j}(x)>0 $ for any $ j=1,2,\dotsc,d 
$, by (P1)  we know that the set $ \bigcup_{x\in E} \{ Y_{k,j}(x)\colon 1\leq j\leq d \}      $ is finite, since  $ \{x[1,k]\colon x\in E  \}  $ is finite. Thus, 
 there exists
 $\mathbf{p}_{i}^{(k)}=(p_{i,1}^{(k)},p_{i,2}^{(k)},\dotsc,p_{i,l_k
}^{(k)})\in \{1,2,\dotsc,{l_k}\}^{l_k} $ 
such that  
\begin{align*}
&\varphi_{p_{i,1}^{(k)}}^{(k)}(x_1x_2\dotsb 
x_k)[1,\psi_{k,1}(x)] =Y_{k,1}(x) ,\\   
&\varphi_{p_{i,2}^{(k)}}^{(k)}(x_1x_2\dotsb 
x_k)[1,\psi_{k,2}(x)]=Y_{k,2}(x)  ,\\
& \dotsc ,  \\
&\varphi_{p_{i,l_k}^{(k)}}^{(k)}(x_1x_2\dotsb 
x_k)[1,\psi_{k,d}(x)]=Y_{k,d}(x)   
\end{align*} for any $ x\in E $ with $ \psi_{k,j}(x)>0 $ for $ j=1,2,\dotsc,d $.  
According 
to the construction of the map $ \Delta_N $, 
there 
exists a positive integer $ 
q_k=r^{(k)}_{i,1} $ 
such that   
\begin{align*}
&\sigma^{q_k}(x)[1,\psi_{k,1}(x)]=Y_{k,1}(x) , \\
& \sigma^{2\cdot 
q_k}(x)[1,\psi_{k,2}(x)]=Y_{k,2}(x), \\ 
&\dotsc  ,  \\
&\sigma^{d\cdot 
q_k}(x)[1,\psi_{k,d}(x)]=Y_{k,d}(x)
\end{align*} 
  hold for any $ x\in E$. 
We claim that  $\{q_k\}_{k=1}^{\infty} $ 
is the sequence that we want.   As we have proved,  for any $ x\in D$ there exists 
$ m(x) >0$ such that $ \psi_{k,j}(x)>0 $ holds 
for any $ k\geq m(x) $, 
which means that both $ 
\sigma^{j\cdot q_k}(x) $ and $ g_j(x) $ 
are contained in the same cylinder
$[Y_{k,j}(x)]$ for $ j=1,2,\dotsc,d $. 
Then according to the fact that $\lim_{k\to\infty} \diam([Y_{k,j}(x)]) =0 $, 
we obtain $\lim_{k\to\infty}\sigma^{j\cdot q_k}(x)=g_j(x) $ 
for $ j=1,2,\dotsc,d $. This ends the proof of Theorem~\ref{main1}.

\begin{remark}
For the full shift over $ N $ symbols, it is 
clear  that $ \rho(\sigma x,\sigma y)= N\rho(x,y) $ 
for any $ x,y\in \Sigma_{N} $ with $\rho(x,y)<\frac{1}{N}$. 
Notice \cite[Lemma 5.1]{FHYZ12} and \cite[Lemma 5.4]{FHYZ12}, 
  it immediately turns out that  the Hausdorff dimension of the multiply Xiong chaotic set $ C $ is equal to its Bowen dimension entropy  divided by the  logarithm  of $N $.  
Note that the Hausdorff dimension is dependent on the metric,
but the  Bowen dimension entropy is not.
\end{remark}

\section{Multiply Xiong chaotic set in the Gauss system}
In this section we first study the Lebesgue measure of scrambled sets in the Gauss system and then provide a proof of Theorem 
~\ref{main2} through the relation between the Gauss system
and the full shift over countable symbols.

\subsection{Proximal cells and Li-Yorke scrambled sets}
Recall that the Gauss map $T\colon [0,1)\to [0,1)$ is defined by
$T(x)=\frac{1}{x}-\bigl\lfloor \frac{1}{x}\bigr\rfloor$ for $x\in (0,1)$ and $T(0)=0$ where $ \lfloor\cdot\rfloor $ represents the integer part of a real number.
The Gauss map induces an infinite continued 
fraction of every irrational number $ x\in [0,1) 
$. Specifically, the continued fraction of $ 
x\in  [0,1)\setminus \mathbb{Q} $ is 
\[ 
x=\dfrac{1}{a_1(x)+\dfrac{1}{a_2(x)+\dfrac{1}{a_3(x)+\dotsb}}} =[a_1(x),a_2(x),a_3(x),\dotsb ] \]
where $a_1(x)=\bigl\lfloor 
\frac{1}{x}\bigr\rfloor  $ and $ 
a_n(x)=\biggl\lfloor \frac{1}{T^{n-1}(x)} 
\biggr\rfloor $ for any $ n\geq 2 $. 
Although the Gauss map is not continuous, it still 
has some interesting dynamical properties  discussed in  the continuous dynamical system. And we are interested in the irrational part mostly. Hence, we call $([0,1)\setminus\mathbb{Q},T)$   the Gauss system and adopt the concepts in the continuous  system to describe the dynamical properties for the Gauss system.

Let $\mathcal{B} $ be the Borel $ \sigma $-algebra on 
$ [0,1) $.    We use $\mathcal{L}$ to denote the Lebesgue measure.
It is well known that the Gauss map 
preserves the Gauss measure $\mu$ that is 
given by  $\mu(A)=\frac{1}{\ln 2}\int_A \frac{1}{1+x}dx$ where the integration is with respect to the Lebesgue measure 
for any $ A\in\mathcal{B} $.
It is clear that the Gauss measure and the Lebesgue measure 
are equivalent, that is they have the same null sets and full measure sets.
It is shown in \cite{R61} that 
the Gauss map with the Gauss measure is exact.

\begin{lem}\label{borel}
	The Gauss map sends Borel sets to Borel sets. 
\end{lem}

\begin{proof}
  Let $ W $ be a Borel subset of $ [0,1) $ and 
  $ W_n =W \cap 
 [\frac{1}{n+1},\frac1n)$ for 
  any $ n\geq 1 $.   Then 
  $ W $ can be written  as $ 
  W=\bigcup_{n=1}^{\infty} 
  W_n\cup ( W\cap \{0\})  $. Observe that $ 
  T|_{[\frac{1}{n+1},\frac1n)} $ is a 
  homeomorphism, $ 
  T(W)=\bigcup_{n=1}^{\infty} 
  T(W_n)\cup T( W\cap \{0\})  $ is a Borel 
  subset of $ [0,1) $.
\end{proof}

\begin{prop}\label{prop4.2}
For 
every irrational number $x $ in the Gauss  system , $ 
LY(x)=\{y\in 
[0,1)\setminus\mathbb{Q} \colon (x,y)$  $\textrm{ is scrambled}\} $
is residual and has full Lebesgue measure. In 
particular, every proximal cell is residual and 
has full Lebesgue measure. 
\end{prop}

\begin{proof}
Now fix a point $x\in [0,1)\setminus\mathbb{Q} $, 
as the Gauss measure $\mu$ is exact, 
 $\mu$ is strongly mixing by Proposition~\ref{prop:exact}; 
In particular, $\mu\times \mu$ is ergodic.
For any $ k\in\mathbb{N} $, applying the well-known Birkhoff 
Ergodic theorem (see \cite{WP}) to the indicator function  $\chi_{[0,1/k]\times[1-1/k,1)} $ in the ergodic system 
$([0,1)\times [0,1),\mathcal{B}\times\mathcal{B},\mu\times\mu, T\times T)$, 
we obtain that for almost every $ (z_1,z_2)\in [0,1)\times [0,1) $,
\[
\lim_{N\to\infty} 
\frac{1}{N}\sum_{i=0}^{N-1} \chi_{[0,1/k]\times[1-1/k,1)}(T^{i}\times T^{i}(z_1,z_2) )=  \mu\times \mu ([0,1/k]\times[1-1/k,1) )>0 . 
\]
It means that  for a fixed positive integer $k$,  there exists a measurable set $ E_k\subset [0,1)\times [0,1) $ with $ \mu\times \mu (E_k)=1 $ such that  for any $ (z_1,z_2)\in E_k $,  
there exists a positive integer $ N_k=N_k(z_1,z_2) $ such that  $ \frac{1}{n}\sum_{i=0}^{n-1} \chi_{[0,1/k]\times[1-1/k,1)}(T^{i}\times T^{i}(z_1,z_2) )>0 $ holds for any $ n> N_k $.  It implies that there exists a positive integer $ 0\leq l_k(z_1,z_2)\leq n-1  $ such that 
$ |T^{l_k}(z_1)-T^{l_k}(z_2)|\geq 1-\frac{2}{k}  $.  Then,  $  \bigcap_{k=1}^{\infty} E_k $ has full 
Lebesgue measure and  is contained in $ 
\{(z_1,z_2): 
\limsup_{n\to\infty}|T^nz_1-T^nz_2|=1 \} $.  
Set
\[
A_m(x)=\Bigl\{y\in[0,1)\setminus\mathbb{Q}\colon 
\limsup_{n\to\infty}|T^nx-T^ny|\geq\frac{1}{m}\Bigr\}.
\]
It is easy to see that the Borel set  $A_m(x)$    is 
  a $G_\delta$
subset of $[0,1)\setminus\mathbb{Q}$ for any 
integer $m\geq 2 $. 
 If $A_m(x)$ does not have full $\mu$-measure,
 then there exist $z_1,z_2\in ([0,1)\setminus A_m(x))\cap \bigcap_{k=1}^{\infty}E_k$ 
 such  that
\[  
 \limsup_{n\to\infty}|T^nz_1-T^nz_2|=1 , \ 
 \limsup_{n\to\infty}|T^nx-T^nz_1|<\dfrac 1m,   \text{ and }   
  \limsup_{n\to\infty}|T^nx-T^nz_2|<\dfrac 1m. 
 \]
This is a contradiction, so $\mu(A_m(x))$ must be equal to $ 1 $ for any $ x\in [0,1) $ and $ m\geq 2 $.
  
Let $ C(x)=\{y\in[0,1)\setminus\mathbb{Q}\colon  \liminf_{n\to\infty}|T^nx-T^ny|=0\}$.
It is easy to see that the Borel set  $C(x)$ is 
 a $G_\delta$ subset of $[0,1)\setminus\mathbb{Q}$.
For $k\geq 1$, let
\[   
C_k=\Bigl\{y\in[0,1)\setminus\mathbb{Q}\colon \liminf_{n\to\infty}
|T^nx-T^ny|>\frac{1}{k}\Bigr\} \] 
and for $m\geq 1$, let
 \[
 C_{k,m}=\Bigl\{y\in[0,1)\setminus\mathbb{Q}\colon  
 |T^nx-T^ny|>\frac{1}{k}+\frac{1}{m},\ \forall n\geq m\Bigr\}.
 \]
It is clear that all $C_k$ and $C_{k,m}$ are 
also Borel sets.  
  To prove  
  $\mu(C(x))=1$, it is sufficient to show that 
  $\mu(C_{k,m})=0$ for all $k,m\geq 1$ since 
  $C_k=\bigcup_{m=1}^\infty C_{k,m}$ and
  $[0,1)\setminus C(x)=\bigcup_{k=1}^\infty 
  C_k$.
Assume by contradiction that there exist some $ 
m,k\geq 1 $ such that $\mu(C_{k,m})>0$.
On one hand, for this fixed $ k\geq 1 $, it is 
easy to compute that the measure of the $ 
\frac{1}{k} $-neighborhood of $T^nx  $ is   \[ 
\mu(B(T^nx,\tfrac1k) 
)=\mu(T^{-n}(B(T^nx,\tfrac1k) )  )>\frac{1}{\ln 
2}\ln \frac{2k+1}{ 2k-1}>0\] for any $ n\geq 1 $. 
On the other hand, 
by Lemma~\ref{borel}, $ 
T^nC_{k,m}\in\mathcal{B} $ for any $ n\geq 1 $ and
$\lim_{n\to\infty} \mu(T^{n}C_{k,m})=1$ from the 
exactness and Proposition~\ref{prop:exact}.
Set $ 
\varepsilon=\frac{1}{\ln 
2}\ln \frac{2k+1}{ 2k-1} $.   
There exists 
some $ N>m $ such that $ 
\mu(T^nC_{k,m})>1-\varepsilon $ for any $ n\geq N 
$. 
Thus, we can pick a point $ z $ in $ 
T^NC_{k,m}\cap 
 B(T^Nx,\frac1k) $. Write $ N=m+l $ with $ l\geq 
 1 $.  There exists  $ y\in T^mC_{k,m} $ such that 
 $ T^ly=z $,  which implies that $ 
 |T^ly-T^{m+l}x|<\frac{1}{k} $. And there also 
 exists $ t\in C_{k,m} $ such that $ y=T^mt $. 
 According to the definition of $ C_{k,m} $, it 
 is easy to verify that \[ 
 |T^{m+l}t-T^{m+l}x|=|T^ly-T^{m+l}x|>\frac{1}{k}+\frac{1}{m},
 \] which is a contradiction. 
Therefore, $ \mu(C_{k,m}) $ must be zero  and   
  $\mu(C(x))=1$.
  Note that $ LY(x) $ is exactly the set $ 
  \bigcap_{m=1}^{\infty}A_m(x)\cap C(x) $,  which 
  has full Lebesgue 
  measure and contains a dense $ G_\delta $ 
  subset. This ends the proof. 
  \end{proof}

Recall that a dynamical system $(X,T)$ is  Li-Yorke sensitive if there exists a sensitive constant $ \delta>0 $ such that for any $ x\in X $ and  any $ \varepsilon>0 $, there exists $ y$ in the $ \varepsilon $-neighborhood $  B(x,\varepsilon) $ such that $ (x,y) $ is proximal and $ \lim\sup_{n\to\infty}d(T^nx,T^ny)\geq \delta $. A scrambled set $S$ is call maximal if
$S$ is maximal in the inclusion relation among 
all scrambled sets. The following two corollaries 
are clear.

\begin{cor}
The Gauss system is Li-Yorke sensitive.
\end{cor}
 
\begin{cor}
Every maximal scrambled set in the Gauss system is uncountable.
\end{cor}
 
\begin{prop}    
Every measurable scrambled  set in the 
Gauss system has Lebesgue measure zero. 
\end{prop}

\begin{proof}
Assume that there exists a scrambled set $B\subset [0,1)\setminus\mathbb{Q}$ 
 with 
 positive Lebesgue measure,  it is clear 
 that $ B $ also has positive Gauss measure.
 Choose a Lebesgue measurable subset $B_1$ of 
 $B$ such that $\mu(B_1)>0$ and
 $\mu(B\setminus B_1)>0$.  
By Lemma~\ref{borel} and 
 the exactness of the Gauss map, we have $ \lim_{n\to\infty} \mu(T^n 
 (B_1))=1$
 and $  \lim_{n\to\infty} \mu(T^n (B\setminus 
 B_1))=1$ for $ 
 n\in\mathbb{N} $.
This means that there exists $N\in\mathbb{N}$ such that 
 $\mu(T^N(B_1))>\frac{1}{2}$ and
 $\mu(T^N (B\setminus B_1))>\frac{1}{2}$.  It  implies   $T^N(B_1)\cap T^N (B\setminus B_1)\neq\emptyset$. 
 This is a contradiction because $B$ is scrambled and $T^N|_B$ must be  injective.
\end{proof}

Note that a multiply Xiong chaotic set is   
scrambled, we have the following corollary.

\begin{cor} 
Every measurable multiply Xiong chaotic set in 
the Gauss system has Lebesgue measure zero.
\end{cor}

\subsection{multiply Xiong chaotic set in the 
full shift over countable symbols}
Let $\Sigma_{\infty}=\mathbb{N}^\infty$.
There exists a natural bijection $\phi\colon \Sigma_{\infty}\to [0,1)\setminus\mathbb{Q}$ by 
$\phi(a_1a_2\dotsc)=[a_1,a_2,\dotsc]$.
Endow $\mathbb{N}$ with the discrete topology.
Then the product space $\Sigma_\infty$ is metrizable, separable, not compact but complete.
A compatible metric $\rho$ on $\Sigma_\infty$ can be 
defined as follow:
for any $x=x_1x_2\dotsb$ and
$y=y_1y_2\dotsb\in\Sigma_\infty$,
\[
\rho(x,y)=\begin{cases}
0,& \text{ if }x=y,\\
\frac{1}{2^k}, &\text{ if }x\neq y \text{ and }
k=\min\{n\in\mathbb{N}\colon x_n\neq y_n\}-1. 
\end{cases}
\] 
It is not hard to show that  $\phi$ is a homeomorphism between $(\Sigma_{\infty},\sigma)$ and $([0,1)\setminus\mathbb{Q},T)$ with $ \phi\circ\sigma=T\circ\phi $.
But it should be noticed that neither $\phi$  nor $\phi^{-1}$ is Lipschitz continuous.
The shift map on $\Sigma_\infty$ is defined by 
$ \sigma\colon \Sigma_{\infty}\to\Sigma_{\infty} $, 
$ \sigma(x)=x_2x_3\dotsb  $ for any $x=x_1x_2\dotsb $.
It is clear that $\sigma$ is continuous.
The pair $ (\Sigma_{\infty}, \sigma) $ 
is called the full shift over countable symbols.
The definitions of word, prefix and 
cylinder etc. are similar to the definitions in 
full shift over finite symbols.

\begin{prop}\label{project}
 If $ E\subset \Sigma_{\infty} $ is a multiply Xiong chaotic set in the full shift over countable symbols $(\Sigma_{\infty},\sigma) $, 
 then $\phi(E) $ is a multiply Xiong chaotic set in the Gauss system.
 Moreover, for any  $ x\in \Sigma_{\infty} $,  if $ E  $ is a subset of $ MProx(x,\sigma) $, then $ \phi(E) $ is contained in $ MProx(\phi(x),T) $.
\end{prop}
\begin{proof}
 For any $ d\in\mathbb{N} $, any non-empty subset $ M\subset \phi(E) 
 $,   and any continuous 
 maps 
 $ g_1,g_2,\dotsc, $  $ g_d:M\to [0,1) $, define  
 continuous maps $ F_j=\phi^{-1}\circ 
 g_j\circ\phi :\phi^{-1}(M)\to \Sigma_{\infty} 
 $ for $ j=1,2,\dotsc, d $.   
 Since $ E $ is a multiply Xiong chaotic set 
 in the full shift over countable symbols, 
 for the non-empty subset $ 
 \phi^{-1}(M)\subset E $ and every continuous 
 map  $ F_j:\phi^{-1}(M)\to \Sigma_{\infty} $ 
 , there exists an increasing sequence $ 
 \{r_n\} $ such that 
 \[
 \lim\limits_{n\to\infty} \sigma^{j\cdot 
  r_n}(\phi^{-1} 
 x)=F_j(\phi^{-1}x)=\phi^{-1}\circ 
 g_j\circ\phi(\phi^{-1} x).   \] It implies 
  \[ \lim\limits_{n\to\infty} \phi\circ 
 \sigma^{j\cdot r_n}(\phi^{-1}x)=T^{j\cdot 
  r_n} (x)=g_j (x) \] for any $  x \in   M  $ 
 and $ j=1,2,\dotsc, d $. The second result is clear. 
\end{proof}
By Proposition~\ref{project}, 
in order to study the multiply Xiong chaotic set in the Gauss system,
we need to construct a multiply Xiong chaotic set in the full shift over countable symbols. 
As the idea is similar to the construction in the proof of 
Theorem~\ref{main1}, we only sketch the construction and the details
are left to interested readers. 

\begin{lem}\cite[Lemma 2]{WT07}\label{jinzhiweishu}
  Let $ W=\prod_{i=1}^{\infty}\{1,2,\dotsc,i\} $ be a compact subset of $ \Sigma_{\infty} $.  We have  $ \dimh(W\cap U)=\infty $ where $ U $ is a cylinder of $ \Sigma_{\infty} $ with $ W\cap U $ not empty. 
\end{lem}

\begin{lem} \label{wuxianweishu} 
 \cite[Lemma 3]{WT07}
 Let $A=\{a_n\}_{n=1}^\infty$ be a sequence of strictly increasing positive numbers.
 Define a map $ \Gamma_{A}:\Sigma_{\infty}\to\Sigma_{\infty} $, 
 $ \Gamma_{A}(x)=x_1x_2\dotsb  x_{a_1-1}x_{a_1+1}\dotsb  
 x_{a_n-1}x_{a_n+1}\dotsb $.
 Let $Y$ be a subset of $\Sigma_{\infty}$.
 If the  density of $A$ is less than $1$
 and $\dimh(\Gamma_{A}(Y))=\infty$, 
 then $ \dimh(Y)=\infty$. 
\end{lem}

\begin{prop}\label{countable}
In the full shift $(\Sigma_\infty,\sigma)$ 
over countable symbols, for every $z\in \Sigma_\infty$, 
the multiply proximal cell of $z$ contains
a multiply Xiong chaotic sets with full Hausdorff dimension everywhere. 
\end{prop}

\begin{proof}

 Let $ n\in\mathbb{N}$ and  $l_n=(n^n)^{n^n} $.  We  list all the 
  self-maps on $ \{1,2,\dotsc,n\}^n $ as $ 
   \varphi^{(n)}_1,\varphi^{(n)}_2,\dotsc,\varphi^{(n)}_{l_n}
    $.  We can define a
     map $ \Delta_{\infty}:\Sigma_{\infty}\to 
    \Sigma_{\infty} $ similar to $ 
    \Delta_{N} $ in the proof of Theorem~\ref{main1}. 
Set $  W=\prod_{i=1}^{\infty}\{1,2,\dots, i \} $.
For any $ n\geq 1 $ and $ 
v=v_1v_2\dotsb v_n\in\mathbb{N}^{n} $, 
we 
can also define a map $ 
\theta_v:\Sigma_{\infty}\to \Sigma_{\infty} 
$ by $ \theta_v(x)=v_1v_2\dotsb v_nx_{n+1}\dotsb $ for any $ x\in \Sigma_{\infty} $. Set $ \{w_n:n\geq 
1\}=\bigcup_{k=  1}^{\infty}\mathbb{N}^k  $, $ 
B_n= W\cap  [12\dotsc n 1 ] $,
$ B=\bigcup_{n= 1}^{\infty}\Delta_{\infty}(B_n) $ and $ C=\bigcup_{n= 1}^{\infty} 
\theta_{w_n}\circ\Delta_{\infty}(B_n) $.  It is easy to see that each $ B_n $ is not empty and pairwise disjoint closed subset, so is $ \Delta_{\infty}(B_n) $ for any $ n\geq 1 $. 

Similarly, we can also define a continuous  map $ \xi\colon B\to C $ such that $ \xi|_{\Delta_{\infty}(B_n)}=\theta_{w_n}|_{\Delta_{\infty}(B_n)} $.
 Obviously, $ C $ is contained in the multiply proximal cell of $ z $.  We proceed to show that $ C $ is with full Hausdorff dimension everywhere and then show its chaotic property.      
For any non-empty open subset $ U $ of $ \Sigma_{\infty} $, there exists some $ w_n $ such that $ [w_n]\subset U $, so $ C\cap U $ contains $ \theta_{w_n}\circ \Delta_{\infty}(B_n) $. 
According to Lemma~\ref{jinzhiweishu},  Lemma~\ref{wuxianweishu}, and the same explanation in the proof of Theorem~\ref{main1}, we know that $ \dimh(\theta_{w_n}\circ \Delta_{\infty}(B_n)) =\infty$. 

The remaining part is to show that $ C $ is a multiply Xiong chaotic set.     Most of this part is similar to the full shift over finite symbols. We only explain the difference here. 
As $ \xi\colon B\to C $ is continuous, it is sufficient to show that $ B $ is a multiply Xiong chaotic set. Let $ E\subset B $ and $ g_j\colon E\to\Sigma_{\infty} $ be a continuous map for $ j=1,2,\dotsc,d $. Similarly, for $ j=1,2,\dotsc, d $, any $ x\in E $,  and any $ k\geq 1 $, we can define an integer $ \psi_{k,j}(x) $ as the proof of Theorem \ref{main1} and all the results about $ \psi_{k,j}(x) $ are also valid.  
Observe that $ B$ is a subset of the compact set $\Delta_{\infty}(W) $, the set $ \{x[1,k]\colon x\in E \} $ is finite. By (P1), we know that the set $ \bigcup_{x\in E }\{Y_{k,j}(x)\colon 1\leq j\leq d \} $ is finite. 
Then we can carry on   applying the 
     method in the proof of Theorem~\ref{main1}
to show that $ C $ 
     is a multiply Xiong chaotic set in 
     $      (\Sigma_{\infty},\sigma) $. 
   
 \end{proof}

\subsection{Proof of Theorem~\ref{main2}}
By Propositions~\ref{project} and~\ref{countable},
we should estimate the Hausdorff dimension of
the image of the multiply Xiong chaotic under the map $\phi$.
To this end, we need some results of 
the continued fraction and we refer the reader 
to  \cite{Io2002}  for more details. 
Assume that the infinite continued fraction of 
$ x \in [0,1)\setminus \mathbb{Q}$  is $ x=[a_1,a_2,\dotsc ] $. 
For any $n\in\mathbb{N}$, we call the rational number 
$[a_1,a_2,\dotsc,a_n] $  the $ n $-th convergent of $ x $ and 
denote it by 
\[ 
\frac{p_n(a_1,a_2,\dotsc ,a_n)}{q_n(a_1,a_2,\dotsc ,a_n)}= [a_1,a_2,\dotsb ,a_n]
\]    
where $ p_n,q_n\in\mathbb{N} $ and  $ (p_n,q_n)=1 $.
If we set $ p_0=q_{-1}=0 $, $ p_{-1}=q_{0}=1 $, then for any 
$ n\in\mathbb{N} $,    
\begin{align*}
 &q_{n}=a_n q_{n-1}+q_{n-2},\\
 &p_{n}=a_n p_{n-1}+p_{n-2}, \\
 &p_nq_{n-1}-p_{n-1}q_n=(-1)^{n+1},  \\      
 &q_n\geq 2^{\frac{n-1}{2}}.
 \end{align*}
For any $ n\in\mathbb{N} $ and $ i_1i_2\dotsc i_n\in\mathbb{N}^n $,   the fundamental interval is defined  by 
\[
I^n(i_1,i_2,\dotsc , i_n)=\{x\in (0,1) : a_j(x)=i_j, 1\leq j\leq n \}
.
\]
The endpoints of $ I^n(i_1,i_2,\dotsc , i_n) $
are $ \frac{p_n+p_{n-1}}{q_n+q_{n-1}} $ and 
$ \frac{p_n}{q_n} $ where $p_n= p_n(i_1,i_2,\dotsc,  i_n) $ 
and $ q_n= q_n(i_1,i_2,\dotsc,  i_n)  $. 
More precisely,
 in the case when $ n $ is odd, 
\[
I^n(i_1,i_2,\dotsc,  i_n) =
\biggl [ 
 \frac{p_n+p_{n-1}}{q_n+q_{n-1}},\frac{p_n}{q_n} 
 \biggr);\]
and in the case  when $ n $ is even, 
\[
 I^n(i_1,i_2,\dotsc , i_n)= \biggl( 
 \frac{p_n}{q_n},\frac{p_n+p_{n-1}}{q_n+q_{n-1}} 
 \biggr] .
 \]
Whatever, the Lebesgue measure of 
 the interval $ I^n(i_1,i_2,\dotsc,  i_n)$ is  
 $\frac{1}{q_n(q_n+q_{n-1}) } $.
We will need the following useful results.

\begin{lem}[\cite{Wu2006}]\label{DZ}
 For any $ n\geq 1 $ and $ 1\leq k\leq n $, it has that \[\frac{a_k+1}{2}\leq \frac{q_n(a_1,a_2,\dotsc ,a_n)}{q_{n-1}(a_1,a_2,\dotsc ,a_{k-1},a_{k+1},\dotsc ,a_n) }\leq a_k+1. \]
\end{lem}

\begin{lem}[\cite{Hu2014}] \label{In-AP}
 There exists a positive number $ \lambda $ such that
 \[\frac{1}{\lambda}|I(\mu)||I(\nu)|\leq |I(\mu\nu)|\leq \lambda|I(\mu)||I(\nu)|  \]
 where $ \mu=u_1u_2\dotsc  
 u_n\in\mathbb{N}^n$ and  $\nu=v_1v_2\dotsc  
 v_k\in\mathbb{N}^k$ for any $ 
 k,n\in\mathbb{N} $.
\end{lem}

\begin{thm}[\cite{J28}] \label{thm:J28}
For any $ k>8 $,
\[1-\frac{4}{k\ln2}\leq \dimh(\phi(\Sigma_k))\leq 1-\frac{1}{8k\ln k}.\]
\end{thm}
 
Inspired by the proof of Theorem~\ref{thm:J28} in \cite{J28},
we have the following result.

 \begin{prop}\label{dim}
 Let $ 
 W=\prod_{i=1}^{\infty}\{1,2,\dotsc,i\} $.
 For any $ l\geq 1 $,  any $ 
 i_1,i_2,\dotsc,i_l\in\mathbb{N} $, and any 
 integer   $ k>\max\{4,i_1,i_2,\dotsc,i_l,l+1\} 
 $,   
\[\dimh(\phi\circ\theta_{i_1i_2\dotsc i_l}(W
 )\cap \phi(\Sigma_k))\geq 1-\frac{4}{k\ln 2} .\]
\end{prop}

\begin{proof}

For any $ n\geq 1 $, denote  \[ 
M^n_{a_1,a_2,\dotsc ,a_n}=\{\theta\in 
[0,1)\setminus\mathbb{Q}: \theta=[a_1,a_2,\dotsb 
,a_n,a_{n+1},\dotsb ] \text{ with } a_{n+1}\leq k \}  . \] 
For any $ n\geq k $, denote   
\[W_n=\bigcup_{a_{l+1}=1}^{l+1}\bigcup_{a_{l+2}=1}^{l+2}\dotsc
\bigcup_{a_k=1}^k\bigcup_{a_{k+1}=1}^k\dotsc 
\bigcup_{a_n=1}^k M^n_{i_1,i_2,\dotsc 
	,i_l,a_{l+1},\dotsc,a_n}  \]
and 
\[V_n=\bigcup_{a_{l+1}=1}^{l+1}\bigcup_{a_{l+2}=1}^{l+2}\dotsc
\bigcup_{a_k=1}^k\bigcup_{a_{k+1}=1}^k\dotsc 
\bigcup_{a_n=1}^k I^n_{i_1,i_2,\dotsc 
	,i_l,a_{l+1},\dotsc,a_n}.  \]
The following properties   are obvious by the definitions:
\begin{enumerate}
	\item $ I^{n+1}_{i_1,i_2,\dotsc 
		,i_l,a_{l+1},\dotsc ,a_{n+1}}\subset 
	I^{n}_{i_1,i_2,\dotsc 
		,i_l,a_{l+1},\dotsc ,a_n}  $  for any $ 
	n\geq k $.
	
	\item $ M^n_{i_1,i_2,\dotsc 
		,i_l,a_{l+1},\dotsc ,a_{n}}\subset  
	I^n_{i_1,i_2,\dotsc 
		,i_l,a_{l+1},\dotsc ,a_{n}} $ for any $ 
	n\geq k $.
	
	\item  If $ a_{n+1}\leq k $, then $ 
	I^{n+1}_{i_1,i_2,\dotsc 
		,i_l,a_{l+1},\dotsc ,a_{n+1}}\subset  
	M^n_{i_1,i_2,\dotsc 
		,i_l,a_{l+1},\dotsc ,a_{n}}  $ for 
	any $ 
	n\geq k $.
	
	\item  $ W_n=V_{n+1} $ for any $ n\geq k $.
	
	\item  $  V_k\supset W_k= V_{k+1}\supset 
	W_{k+1}= V_{k+2} \supset \dotsb  $
\end{enumerate}  
It is clear that   
   \begin{align*}
\phi\circ\theta_{i_1i_2\dotsc i_l}(W)\cap 
\phi(\Sigma_k)=& \{\theta\in 
[0,1)\setminus\mathbb{Q}:    a_j(\theta)=i_j \text { for }1\leq  j\leq l, \\ & a_{j}(\theta)\leq j 
\text{  for  }l< j< k, 
 \text { and } 
a_{j}(\theta) \leq k  \text{  for  }   j\geq  k
\}\\
=&\bigcap_{n=k}^\infty V_n =\bigcap_{n=k}^\infty W_n  ,
\end{align*} which is closed and perfect  
 because  $ V_n $ is not a singular and $ 
\lim_{n\to \infty}\diam( V_n) =0 $. 
It is clear that   $ \phi\circ\theta_{i_1i_2\dotsc i_l}(W)\cap \phi(\Sigma_k) $ is  compact. 
Let $ \delta>0 $ and $ L $ be a positive integer. Assume that   $ 
\mathcal{G}=\{B_1,B_2,\dotsc ,B_L \} $ is a finite $ \delta $-cover  (the diameter of each element of $ \mathcal{G}  $ is less than $ \delta $) of $ \phi\circ\theta_{i_1i_2\dotsc i_l}(W)\cap \phi(\Sigma_k) $ with $ 
B_j\cap \phi\circ\theta_{i_1i_2\dotsc i_l}(W)\cap 
\phi(\Sigma_k)   
$ containing  infinite points for $ j=1,2,\dotsc ,L $. 
In the closed set $ \phi\circ\theta_{i_1i_2\dotsc i_l}(W)\cap \phi(\Sigma_k)  $ we can pick 
  $ b_j=\inf(B_j\cap 
\phi\circ\theta_{i_1i_2\dotsc i_l}(W)\cap 
\phi(\Sigma_k)  ) $ and   $ c_j=\sup(B_j\cap 
\phi\circ\theta_{i_1i_2\dotsc i_l}(W)\cap 
\phi(\Sigma_k)  ) $ for $ 
j=1,2,\dotsc ,L $. 
Thus, $ \mathcal{G}'=\{[b_j,c_j]:j=1,2,\dotsc ,L 
\} $ is still a $ \delta $-cover of $ 
\phi\circ\theta_{i_1i_2\dotsc i_l}(W)\cap 
\phi(\Sigma_k) $ by  replacing $ B_i $ with $ 
[b_i,c_i] $ in $ \mathcal{G }$ for $j=1,2,\dotsc 
,L  $. It is clear that   
\[\sum_{j=1}^{L} (c_j-b_j)^s \leq \sum_{j=1}^L 
 \diam(B_j) ^s\leq \sum_{B\in \mathcal{B}}  \diam(B) ^s  .\]
Since both 
 $ b_j $ and $ c_j $ are elements of   
 \[\phi\circ\theta_{i_1i_2\dotsc i_l}(W)\cap 
 \phi(\Sigma_k)=\bigcap_{n=k}^\infty 
 W_n=\bigcup_{a_{l+1}=1}^{l+1}\bigcup_{a_{l+2}=1}^{l+2}\dotsc
 \bigcup_{a_k=1}^k\bigcup_{a_{k+1}=1}^k\dotsc 
 \bigcup_{a_n=1}^k M^n_{i_1,i_2,\dotsc 
 	,i_l,a_{l+1},\dotsc,a_n}  , \]
 we can set 
 \begin{align*}
 m+1=\max\{&n\geq k+1: \text{ there exist } 1\leq 
 s\neq t\leq k,    
 a_{l+1},a_{l+2},\dotsc ,a_{n-1}  \\
 & \text{ with }  1\leq a_q\leq q    \text{ for 
 }   l+1\leq q\leq k-1,   \text{ and }   a_q\leq 
 k   \text{ for any }  k\leq q\leq n-1   , \\
 & \text{ such that } c_j\in 
 M^n_{i_1,i_2,\dotsc,i_l,a_{l+1},\dotsc 
 	,a_{n-1},s} \text{ and } b_j\in 
 M^n_{i_1,i_2,\dotsc,i_l,a_{l+1},\dotsc,a_{n-1},t} 
 \}. 
 \end{align*}
It is not hard to show that  the following result:
For every $ 
[b_j,c_j]\in \mathcal{G}' $ with $ j=1,2,\dotsc ,L 
$, there exist $ m\geq k $, $ 
a_{l+1},a_{l+2},\dotsc 
,a_m $ with 
$1\leq a_q\leq q $ for $ l+1\leq q\leq k-1 $,  
and $ a_q\leq k $ for any $ k\leq q\leq m $ such 
that $ [b_j,c_j]\subset 
M^m_{i_1,i_2,\dotsc.i_l,a_{l+1},\dotsc ,a_m} 
$.  And there also exist two distinct integers $ s $ and $ t $ with $ 1\leq s\neq  t\leq k $ such that 
$ [b_j,c_j]\cap 
M^{m+1}_{i_1,i_2,\dotsc.i_l,a_{l+1}, 
	,a_m,s}\neq \emptyset  $ and  $ [b_j,c_j]\cap 
M^{m+1}_{i_1,i_2,\dotsc.i_l,a_{l+1},\dotsc 
	,a_m,t}\neq \emptyset  $.

In the following, as only the endpoints of a close interval are needed, we use $ (a,b) $ to denote    $ [a,b) $ and $ [b,a) $
  for convenience.
Observe that  for any $ n\geq k $,  \begin{align*}
M^n_{i_1,i_2,\dotsc,i_l,a_{l+1},\dotsc 
	,a_n}&=\bigcup_{ 1\leq \bigl\lfloor  
	\frac{1}{u}\bigr\rfloor \leq k } 
I^n( i_1,i_2,\dotsc,i_l,a_{l+1},\dotsc 
,a_n+u ) \\
&= \biggl(\frac{(k+1)p_n+p_{n-1} }{ (k+1)q_n+q_{n-1}} , \frac{ p_n+p_{n-1} }{ q_n+q_{n-1}}  \biggr)\\
&=   \biggl(\frac{[(k+1)a_n+1] 
	p_{n-1}+(k+1)p_{n-2} }{ [(k+1)a_n+1] 
	q_{n-1}+(k+1)q_{n-2}} , \frac{(a_n+1) 
	p_{n-1}+p_{n-2} }{ (a_n+1) q_{n-1}+q_{n-2}}  
\biggr)  
\end{align*} where $ p_{n-1}=p_{n-1}(i_1,i_2,\dotsc,i_l,a_{l+1},\dotsc,a_{n-1}) $,  $ q_{n-1}=q_{n-1}(i_1,i_2,\dotsc,i_l,a_{l+1},\dotsc,a_{n-1}) $, $p_{n}=p_{n}(i_1,i_2,\dotsc,i_l,a_{l+1},\dotsc,a_{n}) $, and $q_{n}=q_{n}(i_1,i_2,\dotsc,i_l,a_{l+1},\dotsc,a_{n}) $. 
As a result,  the distance between $
M^{m+1}_{i_1,i_2,\dotsc,i_l,a_{l+1},\dotsc,a_{m},s}
$ and $ 
M^{m+1}_{i_1,i_2,\dotsc,i_l,a_{l+1},\dotsc 
	,a_{m},t} $ is not less than the minimum of 
\begin{equation}\label{ss}
\biggl| \frac{(t+1) p_{m}+p_{m-1} }{ (t+1) 
	q_{m}+q_{m-1}}- \frac{[(k+1)s+1] 
	p_{m}+(k+1)p_{m-1} }{ [(k+1)s+1] 
	q_{m}+(k+1)q_{m-1}}\biggr|
\end{equation}
and
\begin{equation}\label{tt}
\biggl| \frac{(s+1) p_{m}+p_{m-1} }{ (s+1) 
	q_{m}+q_{m-1}}- \frac{[(k+1)t+1] 
	p_{m}+(k+1)p_{m-1} }{ [(k+1)t+1] 
	q_{m}+(k+1)q_{m-1}}\biggr|
\end{equation}
Furthermore, both \eqref{ss} and \eqref{tt} are not less than $ \frac{1}{ 3k^3 q_{m}(q_{m}+q_{m-1})  }$ , because
\begin{align*}
&\biggl| \frac{(t+1) p_{m}+p_{m-1} }{ (t+1) p_{m}+p_{m-1}}- \frac{[(k+1)s+1] p_{m}+(k+1)p_{m-1} }{ [(k+1)s+1] q_{m}+(k+1)q_{m-1}}\biggr|\\
= & \biggl| \frac{(k+1)(t+s+1)+1 }{   [(t+1)q_{m}+q_{m-1} ] [(k+1)(sq_{m}+q_{m-1})+q_{m} ]  } \biggr| \\
\geq & \frac{1}{ 3k^3 q_{m}(q_{m}+q_{m-1})  }
\end{align*}
and similarly   
\begin{align*}
&\biggl| \frac{(s+1) p_{m}+p_{m-1} }{ (s+1) 
	q_{m}+q_{m-1}}- \frac{[(k+1)t+1] 
	p_{m}+(k+1)p_{m-1} }{ [(k+1)t+1] 
	q_{m}+(k+1)q_{m-1}}\biggr|\\
\geq &  
\frac{1}{ 3k^3 q_{m}(q_{m}+q_{m-1})  }   .
\end{align*}  Therefore, it can be concluded 
that   \[ c_j-b_j\geq \frac{1}{3k^3}\bigl| 
I^{m}( i_1,i_2,\dotsc i_l,\dotsc,a_{l+1},\dotsc 
,a_{m})  \bigr| \]  from the  
fact that   \[\bigl| I^{m}( i_1,i_2,\dotsc 
i_l,\dotsc,a_{l+1},\dotsc ,a_{m})\bigr| =  
\frac{1}{   
	q_{m}(q_{m}+q_{m-1})  }  .\] So far, it has been 
proved that for any $ [b_j,c_j]\in \mathcal{G}'  
$, there exists an interval $ I^m( i_1,i_2,\dotsc 
i_l,$  $\dotsc,a_{l+1},\dotsc ,a_{m}) $ corresponding 
to it where $ m $ and $ 
a_{l+1},a_{l+2},\dotsc ,a_m $ are all related to 
$ 
[b_j,c_j] $.  Next, replacing $ [b_j,c_j] $ by the 
interval  $ I^m( i_1,i_2,\dotsc 
i_l,\dotsc,a_{l+1},\dotsc ,a_{m}) $,   we 
get another family $ \mathcal{C} $ that is   a   
$ 3k^3 \delta $-cover  of $ 
\phi\circ\theta_{i_1i_2\dotsc i_l}(W)\cap 
\phi(\Sigma_k) 
$.

Hitherto, we have known that for any $ p>0 $ and 
any  finite $ \delta
$-cover $ \mathcal{G} $ of $ 
\phi\circ\theta_{i_1i_2\dotsc i_l}(W)\cap 
\phi(\Sigma_k) $, there exists a $ 3k^3\delta 
$-cover 
$ \mathcal{C} $ of $ 
\phi\circ\theta_{i_1i_2\dotsc i_l}(W)\cap 
\phi(\Sigma_k) 
$   satisfying    \[ 3^sk^{3s}\sum_{B\in 
	\mathcal{G}} \diam(B) ^s \geq 
\sum_{C\in\mathcal{C}} \diam(C) ^s  \text{ for  any }   s>0  . \] 
Moreover, set  
\begin{multline*}
h=  \min\{ n : \text{ there 	exists   } n\text{ and } 
a_{l+1},a_{l+2},\dotsc,a_n\in\mathbb{N} \\ \text{ such that some }   I^n 
( i_1,i_2,\dotsc,i_l,a_{l+1},\dotsc,a_n )\in 
\mathcal{C}  \}.  
\end{multline*}
Obviously, $h \geq k   $.
If  there exist two intervals, for example some 
$ I^r , I^n\in\mathcal{C} $ with $ r>n $ and 
$I^r\subset I^n  $, then remove $ I^r $ from $ 
\mathcal{C} $.  We can  get another $ 
3k^3\delta $-cover $ \mathcal{M} $ by such a 
further repetition. It is clear that   any 
two distinct intervals $ I^r $ and $ I^n $ in $ 
\mathcal{M} $  do not intersect. Set   \[p=\max\{n: \text{ there 
	exists some  } I^n\in \mathcal{M}  \} .  \] Then $ p\geq h\geq k $.   
According to the definition, there exist  $ 
a_{l+1},a_{l+2},\dotsc ,a_p $ with $ a_i \leq i $ for $ l+1\leq i\leq k-1 $
 and $a_k, a_{k+1},\dotsc 
,a_p\leq k $ such that $ 
I^p ( i_1,i_2,\dotsc,i_l,a_{l+1},\dotsc 
,a_p )\in\mathcal{M} $, which implies that
$ 
I^p ( i_1,i_2,\dotsc,i_l,a_{l+1},\dotsc 
,1 )$,  $I^p ( i_1,i_2,\dotsc,i_l,a_{l+1},\dotsc 
,2 )$, $\dotsc$, 
$I^p ( i_1,i_2,\dotsc,i_l,a_{l+1},$  $\dotsc 
,k ) $ are certainly all in $ 
\mathcal{M} $.

\medskip
\noindent\textbf{Claim. } Let $ s=1-\frac{4}{k\ln 2} $. 
For any $ n\geq h+1 $, any  $ 
a_{l+1},a_{l+2},\dotsc ,a_p $ with $ a_i \leq i $ for $ l+1\leq i\leq k-1 $, and any  $ a_k,a_{k+1},\dotsc ,a_p\leq k $, 
it has that 
\[|I^{n-1}( i_1,i_2,\dotsc,i_l,a_{l+1},\dotsc 
,a_{n-1} )|^s\leq 
\sum_{i=1}^{k} |I^n( 
i_1,i_2,\dotsc.i_l,a_{l+1},\dotsc ,a_{n-1},i ) |^s.\]

\medskip
\noindent\textbf{Proof of the claim. }
At first, compute    \begin{align*}
\sum_{i=1}^{k} |I^n( 
i_1,i_2,\dotsc.i_l,a_{l+1},\dotsc ,a_{n-1},i ) | 
&= \frac{1}{q_{n-1}(q_{n-1}+q_{n-2}) }\biggl(1-\frac{q_{n-1}+q_{n-2}}{(k+1)q_{n-1}+q_{n-2} } \biggr).
\end{align*}
 So,
\begin{align*}
& \sum_{i=1}^{k} |I^n( 
i_1,i_2,\dotsc.i_l,a_{l+1},\dotsc ,a_{n-1},i ) 
|^s\\
=& \sum_{i=1}^{k}\frac{1}{(iq_{n-1}+q_{n-2} )[(i+1)q_{n-1}+q_{n-2} ] } \times (iq_{n-1}+q_{n-2})^{1-s}[(i+1)q_{n-1}+q_{n-2} ]^{1-s}\\
\geq & \frac{1}{q_{n-1}(q_{n-1}+q_{n-2}) }\biggl(1-\frac{q_{n-1}+q_{n-2}}{(k+1)q_{n-1}+q_{n-2} } \biggr) \times 2^{1-s}q_{n-1}^{1-s}(q_{n-1}+q_{n-2})^{1-s}\\
\geq &\frac{1}{q_{n-1}(q_{n-1}+q_{n-2}) }  \biggl( 1-\frac{\beta}{k}\biggr) \times 2^{1-s}q_{n-1}^{1-s}(q_{n-1}+q_{n-2})^{1-s} \\
=& |I^{n-1}  (i_1,i_2,\dotsc,i_l,a_{l+1},\dotsc 
,a_{n-1} ) |^s\biggl( 
1-\frac{\beta}{k}\biggr) \times 2^{1-s}
\end{align*}
where $ 
\beta=\dfrac{kq_{n-1}+kq_{n-2}}{kq_{n-1}+q_{n-1}+q_{n-2}
} $ and $ \dfrac{1}{2}<\beta<2 $.
If $ s=1-\dfrac{4}{k\ln 2} $, then  \begin{align*}
-\ln \biggl(1-\frac{2}{k} \biggr)
&<\frac{2}{k}+\frac{1}{k}\biggl(\frac{2^2}{k}+\frac{2^3}{k^2}+\dotsc +\frac{2^n}{k^{n-1}}+\dotsc  \biggr) \\
&<\frac{4}{k}  =(1-s)\ln 2. 
\end{align*}
We obtain that  $ \biggl(1-\frac{\beta}{k}\biggr)\times 2^{1-s} \geq 1 $ for any $ \frac{1}{2}<\beta<2 $.
So, it is allowed to replace  the following $ k $ 
intervals : 
\begin{align*}
&I^k(i_1,i_2,\dotsc,i_l,a_{l+1},\dotsc ,a_{k-1},1) , \\
&I^k(i_1,i_2,\dotsc,i_l,a_{l+1},\dotsc 
,a_{k-1},2 
),\\
&\dotsc, \\
& I^k(i_1,i_2,\dotsc,i_l,a_{l+1},\dotsc 
,a_{k-1},k 
) 
\end{align*}
by 
one interval $ 
I^{k-1}(i_1,i_2,\dotsc,i_l,a_{l+1},\dotsc 
,a_{k-1}) $. Do such 
replacement until we get a $ 3k^3\delta $-cover $ 
\mathcal{J} $ in which the longest interval is 
some $ I^h $.  Through the process, it 
guarantees that 
\[ \sum_{A\in 
	\mathcal{M}}|A|^s\geq \sum_{E\in\mathcal{J}} 
|E|^s>0 .\] Thus, for any positive number $ p<\delta $ and any $ p $-cover  
$\mathcal{B} $, it has that 
\[\mathcal{H}^s(\mathcal{B})\geq \sum_{B\in 
	\mathcal{B}}|B|^s\geq 
\frac{1}{3^sk^{3s}}\sum_{C\in\mathcal{C}}|C|^s\geq\frac{1}{3^sk^{3s}}\sum_{A\in\mathcal{M}}|A|^s\geq
\frac{1}{3^sk^{3s}}\sum_{E\in\mathcal{J}} 
|E|^s>0,  \] 
which means  $ 
\dimh(\phi\circ\theta_{i_1i_2\dotsc i_l}(W)\cap 
\phi(\Sigma_k))\geq s= 1-\frac{4}{k\ln 2} $.
\end{proof}

Now we are ready to prove Theorem~\ref{main2}.
 \begin{proof} [Proof of Theorem~\ref{main2}]
Let $ 
W=\prod_{i=1}^{\infty}\{1,2,\dotsc,i\} $.  According to the proof of Proposition~\ref{countable}, we can construct a multiply Xiong chaotic set $ \bigcup_{n=1}^{\infty}\theta_{w_n}\circ\Delta_{\infty}(W) $ contained in the multiply proximal cell of $ \phi^{-1}(z) $.  By 
Proposition~\ref{project},  let    $  
  L=\phi 
  (C)=\bigcup_{n=1}^{\infty} \phi(
   \theta_{w_n}\circ\Delta_{\infty}(W) 
  )  $ be a 
  multiply
  Xiong chaotic set contained in $ MProx(z,T) $  
  in 
  the Gauss system. Now, it remains to  
  show that $ L\subset  [0,1) $ is a subset with 
  full Hausdorff dimension everywhere.   
  Note that  $ \phi \circ\theta_{i_1i_2\dotsb i_n}\circ \Delta_{\infty}(W) 
  $ is exactly the set $L\cap I^n(i_1,i_2,\dotsc ,i_n)  $ for any $ n\geq 1  $ and  any  
  $ i_1,i_2,\dotsb,i_n\in \mathbb{N}$, 
  it is 
  sufficient to show that 
  $ \dimh(\phi\circ\theta_{i_1i_2\dotsb i_n}\circ  \Delta_{\infty}(W)) =1 $. 
We divide the remaining part of the proof into three parts.   
First, define a map
\begin{align*}
f:&\phi\circ \theta_{i_1i_2\dotsc i_n}\circ 
\Delta_{\infty}(W)\to \phi\circ 
\theta_{i_1i_2\dotsc i_n} (W) ,\\
&\phi\circ\theta_{i_1i_2\dotsb 
 i_n}\circ\Delta_{\infty}(x) \mapsto  [ 
i_1,i_2,\dotsb,i_n,x_{n+1},\dotsb    ] . 
\end{align*}
It is not hard to see that 
$  f=\phi\circ\theta_{i_1i_2\dotsc i_n}\circ 
   \Delta_{\infty}^{-1}\circ\theta_{i_1i_2\dotsb
     i_n}^{-1}\circ\phi^{-1}$
 and $ f $ is bijective and  continuous.

Second, recall the subset $ D $ of $\mathbb{N} $ with zero density  in the construction of the map 
$ \Delta_{\infty} $ in the proof of Theorem \ref{main1}.
Let $\tau(n) =\#(\{a_i\leq n:a_i\in  D \} )  $.  It can be inferred that 
$  \lim\limits_{n\to\infty} 
 \frac{1+c\cdot\tau(n)}{n-\tau(n)-1 } =0 $ 
 where $ c $ is a non-negative constant. 
 For a fixed real number $ k>0 $, 
  if we choose 
 $ c=\frac{2\ln(k+1)}{\ln 2} $, 
 then for any $  \varepsilon>0 $, there exists a large enough 
   $ N \in\mathbb{N}$ such that for any $ n\geq N $,  
 \[
 \frac{\frac{2\ln  (k+1)}{\ln 2}\cdot \tau(n) 
 +1}{n-\tau(n)-1}<\varepsilon, 
\] 
  which means 
 \[
 \biggl(  2^{\varepsilon(n-\tau(n)-1)}\biggr)^{\frac{1}{1+\varepsilon}}
 \geq\biggl(   2(k+1)^{2\tau(n)}\biggr)^{\frac{1}{1+\varepsilon}}.
 \]
     
Third, we are going to prove that for $ \varepsilon>0 $,   $ f $ satisfies the locally $ 
   \frac{1}{1+\varepsilon} $-H\"older condition. 
 Since $ \Delta_{\infty}(W) $ is compact, 
 the set 
 $\mathcal{D}_i =\{x[i]:x\in \Delta_{\infty}(W)  \} $ 
 is finite for any  $ i\geq 1 $. 
Set 
\[
  r=\min_{a_i\in \mathcal{D}_i,\  n+1\leq i\leq N}
  |I^N(i_1,i_2,\dotsc,i_n,a_{n+1},\dotsc, a_N ) |.
  \]
For a pair of two distinct points 
$ \beta=[i_1,i_2,\dotsb, i_n,b_{n+1},\dotsb]$
and $\gamma=[i_1,i_2,\dotsb,i_n,c_{n+1},\dotsc ]$  in 
$ \phi\circ \theta_{i_1i_2\dotsc i_n}\circ\Delta_{\infty}(W) $  
with $ |\beta-\gamma|<r  $, 
   there 
 exists some $ m \geq N $ such that 
 $b_i=c_i $ for any $ 1\leq i\leq m $ but 
 $ b_{m+1}\neq c_{m+1} $ with   $  N+1\leq m+1\not \in D $. 
Obviously, both $ b_{m+1} $ and $ c_{m+1} $ are not more than 
some $ k $ since 
$ \phi\circ \theta_{i_1i_2\dots i_n}\circ\Delta_{\infty}(W) $ 
is  compact. 
If $ m $ is odd, $ \gamma$ is less than $\beta $ and the interval 
$I^{m+2}(b_1,b_2,\dotsc ,b_{m+1},k+1)  $ is 
contained in $ [\gamma,\beta] $. 
Otherwise, $ \gamma$ is larger than $\beta $ and $I^{m+2}(b_1,b_2,\dotsc ,b_{m+1},k+1) $ is contained in $ [\beta,\gamma] $. 
So, by Lemma~\ref{In-AP} ,
 \begin{align} \label{a}
|\gamma-\beta| \geq | I^{m+2}(b_1,b_2,\dotsc ,b_{m+1},k+1)| 
 &\geq \frac{| I^{m+1} ( b_1,b_2,\dotsc ,b_{m+1})||I^1(k+1) | }{\lambda} \\\notag
&\geq\frac{|I^{m } ( b_1,b_2,\dotsc ,b_{m })|
 |I^1(b_{m+1})| |I^1 ( k+1)| }
{\lambda^ 2 } \\ \notag
 &\geq \frac{|I^{m }( b_1,b_2,\dotsc ,b_{m })|
   |I^1(k+1)| ^2 }{\lambda^ 2 }.\notag
 \end{align} 
Set  $\eta$ be the word by erasing the items with 
position in $ D $ from $ i_1,\dotsb ,i_n,b_{n+1},\dotsb, b_{m}$. Let $  M_{i_1,i_2,\dotsc,i_n,\dotsc,b_m}=
\max\{i_1,i_2,\dotsc,i_n,\dotsc ,b_m\}+1 $.  Note that $ m\geq N $, 
by Lemma~\ref{DZ},   
\begin{align*}
 | f(\gamma)-f(\beta)|  \leq  |I^{N-\tau(N)}(\eta )   | 
 &\leq q^{-2}_{m-\tau(m)} (\eta )   
 =q^{-\frac{2}{1+\varepsilon} }_{m-\tau(m)} (\eta) \times 
q^{-\frac{2\varepsilon}{1+\varepsilon}  }_{m-\tau(m)} (\eta ) \\
&\leq q^{-\frac{2}{1+\varepsilon} }_{m-\tau(m)} 
(\eta)\times \bigl(2^{\varepsilon(m-\tau(m)-1)}\bigr)^{-\frac{1}{1+\varepsilon}}
\\
& \leq  q^{-\frac{2}{1+\varepsilon} }_{m-\tau(m)} (\eta)\times \bigl( 2M_{i_1,i_2,\dotsc ,i_n,b_{n+1},\dotsc,b_m}^{2\tau(m)}\bigr)^{-\frac{1}{1+\varepsilon}}\\
 & \leq q_m^{-\frac{2}{1+\varepsilon}}(b_1b_2\dotsc b_m) \\
 &\leq |I^m(b_1,b_2,\dotsc ,b_m)|^{\frac{1}{1+\varepsilon}}  .
\end{align*}
 Combine with \eqref{a}, it is clear that 
 \[ 
|f(\gamma)-f(\beta) |\leq 
  \bigl(\frac{\lambda^2}{|I^1 (k+1) | ^2} 
   \bigr)^{\frac{1}{1+\varepsilon}}
   |\gamma-\beta|^{\frac{1}{1+\varepsilon}} . 
  \] 
It can be concluded from Lemma~\ref{holder}
 that \[ \dimh (\phi\circ\theta_{i_1i_2\dotsc i_n} (W) )
 \leq \dimh (\phi\circ\theta_{i_1i_2\dotsc i_n}\circ\Delta_{\infty} (W)) \] for any $i_1,i_2,\dotsc,i_n\in\mathbb{N} $.
Hence,  
\[\dimh(L\cap I^n(i_1i_2\dotsc i_n))\geq \dimh
(\phi\circ\theta_{i_1i_2\dotsc i_n} (W)) \]
for any $ i_1,i_2,\dotsc,i_n\in\mathbb{N} $.
 By Proposition~\ref{dim},  
 \[\dimh
(\phi\circ\theta_{i_1i_2\dotsc i_n} (W)) \geq  \dimh(\phi\circ\theta_{i_1i_2\dotsc i_l}(W
)\cap \phi(\Sigma_k)) \geq 1-\frac{4}{k\ln 2}. \] We obtain  
 \[
 1\geq  \dimh (L\cap I^n(i_1i_2\dotsc i_n) ) \geq 1-\frac{4}{k\ln 2}\] 
 for any $k>\max\{i_1,i_2,\dotsc,i_n,n+1\} $. 
Passing $ k\to \infty $, this implies  
\[\dimh (L\cap I^n(i_1i_2\dotsc i_n) )=1 \]
for any $ i_1,i_2,\dotsc,i_n\in\mathbb{N} $. 
 \end{proof}

\section*{Acknowledgments}
 The work was supported by NNSF of China (Nos. 11471125, 11871228 and 11771264) and NSF of Guangdong Province(2018B030306024).


\begin{bibsection}
\begin{biblist}[\resetbiblist{99}]
  
  \bib{MR2718894}{article}{
  	author={Akin, E.},
  	author={Glasner, E.},
  	author={Huang, W.},
  	author={Shao, S.},
  	author={Ye, X.},
  	title={Sufficient conditions under which a transitive system is chaotic},
  	journal={Ergodic Theory Dynam. Systems},
  	volume={30},
  	date={2010},
  	number={5},
  	pages={1277--1310},
  	issn={0143-3857},
  	review={\MR{2718894}},
  }
  
  
\bib{Au60}{article}{
  author={Auslander, Joseph},
  title={On the proximal relation in topological dynamics},
  journal={Proc. Amer. Math. Soc.},
  volume={11},
  date={1960},
  pages={890--895},
  issn={0002-9939},
  review={\MR{0164335}},
  }
 
 \bib{AK03}{article}{
  author={Akin, Ethan},
  author={Kolyada, Sergi\u\i },
  title={Li-Yorke sensitivity},
  journal={Nonlinearity},
  volume={16},
  date={2003},
  number={4},
  pages={1421--1433},
  issn={0951-7715},
  review={\MR{1986303}},
 }

\bib{BL99}{article}{
	author={Balibrea, Francisco},
	author={Jim\'{e}nez L\'{o}pez, V\'{i}ctor},
	title={The measure of scrambled sets: a survey},
	journal={Acta Univ. M. Belii Ser. Math.},
	number={7},
	date={1999},
	pages={3--11},
	issn={1338-712X},
	review={\MR{1766951}},
}

\bib{BH08}{article}{
	author={Blanchard, Fran\d{c}ois},
	author={Huang, Wen},
	title={Entropy sets, weakly mixing sets and entropy capacity},
	journal={Discrete Contin. Dyn. Syst.},
	volume={20},
	date={2008},
	number={2},
	pages={275--311},
	issn={1078-0947},
	review={\MR{2358261}},
}


\bib{BHS08}{article}{
	author={Blanchard, Fran\d{c}ois},
	author={Huang, Wen},
	author={Snoha, L'ubom\'{i}r},
	title={Topological size of scrambled sets},
	journal={Colloq. Math.},
	volume={110},
	date={2008},
	number={2},
	pages={293--361},
	issn={0010-1354},
	review={\MR{2353910}},
}




\bib{B09}{article}{
 author={Bruin, Henk},
 author={Jim\'enez L\'opez, V\'\i ctor},
 title={On the Lebesgue measure of Li-Yorke pairs for interval maps},
 journal={Comm. Math. Phys.},
 volume={299},
 date={2010},
 number={2},
 pages={523--560},
 issn={0010-3616},
 review={\MR{2679820}},
}


\bib{FHYZ12}{article}{
	author={Fang, Chun},
	author={Huang, Wen},
	author={Yi, Yingfei},
	author={Zhang, Pengfei},
	title={Dimensions of stable sets and scrambled sets in positive finite
		entropy systems},
	journal={Ergodic Theory Dynam. Systems},
	volume={32},
	date={2012},
	number={2},
	pages={599--628},
	issn={0143-3857},
	review={\MR{2901362}},
}


  
   \bib{Fu81}{book}{
    author={Furstenberg, H.},
    title={Recurrence in ergodic theory and combinatorial number theory},
    note={M. B. Porter Lectures},
    publisher={Princeton University Press, Princeton, N.J.},
    date={1981},
    pages={xi+203},
    isbn={0-691-08269-3},
    review={\MR{603625}},
   }
 
 
   \bib{Fa2003}{book}{
   author={Falconer, Kenneth},
   title={Fractal geometry},
   edition={3},
   note={Mathematical foundations and applications},
   publisher={John Wiley \& Sons, Ltd., Chichester},
   date={2014},
   pages={xxx+368},
   isbn={978-1-119-94239-9},
   review={\MR{3236784}},
  }
  
  \bib{Hu2014}{article}{
   author={Hu, Hui},
   author={Yu, Yueli},
   title={On Schmidt's game and the set of points with non-dense orbits
    under a class of expanding maps},
   journal={J. Math. Anal. Appl.},
   volume={418},
   date={2014},
   number={2},
   pages={906--920},
   issn={0022-247X},
   review={\MR{3206688}},
  }
 
 
 \bib{HLYZ17}{article}{
 	author={Huang, Wen},
 	author={Li, Jian},
 	author={Ye, Xiangdong},
 	author={Zhou, Xiaoyao},
 	title={Positive topological entropy and $\Delta$-weakly mixing sets},
 	journal={Adv. Math.},
 	volume={306},
 	date={2017},
 	pages={653--683},
 	issn={0001-8708},
 	review={\MR{3581313}},
 }
 

  
  \bib{HW04}{article}{
   author={Huang, Wen},
   author={Shao, Song},
   author={Ye, Xiangdong},
   title={Mixing and proximal cells along sequences},
   journal={Nonlinearity},
   volume={17},
   date={2004},
   number={4},
   pages={1245--1260},
   issn={0951-7715},
   review={\MR{2069703}},
  }

 
 \bib{Io2002}{book}{
  author={Iosifescu, Marius},
  author={Kraaikamp, Cor},
  title={Metrical theory of continued fractions},
  series={Mathematics and its Applications},
  volume={547},
  publisher={Kluwer Academic Publishers, Dordrecht},
  date={2002},
  pages={xx+383},
  isbn={1-4020-0892-9},
  review={\MR{1960327}},
 }
  
  \bib{J28}{article}{
   author={{Jarn\' ik}, {Vojt\v ech}},
   title={Zur metrischen Theorie der diophantischen Approximationen},
   language={German},
   journal={Prace Matematyczno-Fizyczne},
   volume={36},
   date={1928},
   url={http://eudml.org/doc/215473},
   language = {pol},
   issue = {1},
    number={1},
   pages={91--106},
    issn={0025-5874},
    review={\MR{1544787}},
  }
 
 
 \bib{Liu2019}{article}{
	author={Liu, Kairan},
	title={$\Delta $-weakly mixing subset in positive entropy actions of a nilpotent group},
	journal={J. Differential Equations},
	date={2019},
	issn={0022--0396},
 	doi={doi.org/10.1016/j.jde.2019.01.018},
 	url={www.sciencedirect.com/science/article/pii/S0022039619300403},
}



  
  
  \bib{Li2017}{article}{
   author={Liu, Weibin},
   author={Li, Bing},
   title={Chaotic and topological properties of continued fractions},
   journal={J. Number Theory},
   volume={174},
   date={2017},
   pages={369--383},
   issn={0022-314X},
   review={\MR{3597396}},
  }
 
 
 
 
 
 
 \bib{R61}{article}{
  author={Rohlin, V. A.},
  title={Exact endomorphisms of a Lebesgue space},
  language={Russian},
  journal={Izv. Akad. Nauk SSSR Ser. Mat.},
  volume={25},
  date={1961},
  pages={499--530},
  issn={0373-2436},
  review={\MR{0143873}},
 }
  
  
  
  \bib{WT07}{article}{
   author={Wu, ChunLan},
   author={Tan, Feng},
   title={Hausdorff dimension of a chaotic set of shift in a countable
    symbolic space},
   language={Chinese, with English and Chinese summaries},
   journal={J. South China Normal Univ. Natur. Sci. Ed.},
   date={2007},
   number={4},
   pages={11--16},
   issn={1000-5463},
   review={\MR{2385049}},
  }
 
 \bib{WP}{book}{
  author={Walters, Peter},
  title={An introduction to ergodic theory},
  series={Graduate Texts in Mathematics},
  volume={79},
  publisher={Springer-Verlag, New York-Berlin},
  date={1982},
  pages={ix+250},
  isbn={0-387-90599-5},
  review={\MR{648108}},
 }
  
  \bib{Wu2006}{article}{
   author={Wu, Jun},
   title={A remark on the growth of the denominators of convergents},
   journal={Monatsh. Math.},
   volume={147},
   date={2006},
   number={3},
   pages={259--264},
   issn={0026-9255},
   review={\MR{2215567}},
  } 
  
  \bib{X91}{article}{
   author={Xiong, JinCheng},
   author={Yang, ZhongGuo},
   title={Chaos caused by a topologically mixing map},
   conference={
    title={Dynamical systems and related topics},
    address={Nagoya},
    date={1990},
   },
   book={
    series={Adv. Ser. Dynam. Systems},
    volume={9},
    publisher={World Sci. Publ., River Edge, NJ},
   },
   date={1991},
   pages={550--572},
   review={\MR{1164918}},
  }
  
  \bib{XX91}{article}{
   author={Xiong, JinCheng},
   title={Erratic time dependence of orbits for a topologically mixing map},
   language={English, with Chinese summary},
   journal={J. China Univ. Sci. Tech.},
   volume={21},
   date={1991},
   number={4},
   pages={387--396},
   issn={0253-2778},
   review={\MR{1155316}},
  }
  
  \bib{X95}{article}{  
   author={Xiong, JinCheng},
   title={Hausdorff dimension of a chaotic set of shift of a symbolic space},
   journal={Sci. China Ser. A},
   volume={38},
   date={1995},
   number={6},
   pages={696--708},
   issn={1001-6511},
   review={\MR{1351235}},
  }  
 \end{biblist}
\end{bibsection}
\end{document}